\documentclass[12pt]{amsart}

\usepackage{amssymb}
\usepackage{pdfsync}
\usepackage{cite}
\usepackage{slashbox}	
\usepackage{amsfonts}	
\usepackage{amstext}
\usepackage{amsthm,amsmath,amsfonts,amssymb,amscd,latexsym,txfonts,stmaryrd}
\usepackage{color}
\usepackage{epsfig}
\usepackage[all,cmtip]{xy}
\usepackage{enumerate}

\newtheorem{theorem}{Theorem}[section]
\newtheorem*{theoremd*}{Definition/Theorem}
\newtheorem*{theorem1*}{Theorem 1}
\newtheorem*{theorem*}{Theorem 1}
\newtheorem*{problem*}{Problem}

\newtheorem*{question*}{Question}

\newtheorem{definition}{Definition}[subsection]
\newtheorem*{remarks*}{Remarks}
\newtheorem*{claim*}{Claim}

\newtheorem*{remark*}{Remark}

\newtheorem*{hlt*}{Hard Lefschetz Theorem}
\newtheorem*{HRR*}{Hodge-Riemann Bilinear Relations}
\newtheorem*{basisthm*}{Basis Theorem}
\newtheorem*{relbasisthm*}{Relative Basis Theorem}
\newtheorem*{primdecomp*}{Primitive Decomposition Theorem}
\newtheorem{proposition}{Proposition}[subsection]
\newtheorem{lemma}{Lemma}[subsection]
\newtheorem{corollary}{Corollary}[subsection]
\newtheorem*{corollary1*}{Corollary 1}
\newtheorem*{corollary*}{Corollary}
\newcommand{\p}{{\mathbf{p}}}
\newcommand{\y}{{\mathbf{y}}}

\newcommand{\Dc}{\Delta_c^+}
\newcommand{\Dcp}{\Delta_{c'}^-}

\newcommand{\taupp}{\tau_{c+}}
\newcommand{\taupm}{\tau_{c-}}

\newcommand{\dsp}{\displaystyle}
\newcommand{\Res}{\operatorname{Res}}
\newcommand{\Lie}{\operatorname{Lie}}

\newcommand{\Kirc}{\mathcal{K}_{c}}
\newcommand{\Kircp}{\mathcal{K}_{c'}}
\newcommand{\Kir}{\mathcal{K}}

\newcommand{\xymap}{\xymatrixcolsep{5pc}\xymatrixrowsep{1pc}\xymatrix}

\newcommand{\Maps}{\operatorname{Maps}}
\newcommand{\supp}{\operatorname{supp}}
\newcommand{\ind}{\operatorname{ind}}

\newcommand{\sra}{\shortrightarrow}

\newcommand{\Sym}{\operatorname{Sym}}

\newcommand{\R}{{\mathbb R}}
\newcommand{\Z}{{\mathbb Z}}

\newcommand{\N}{{\mathbb N}}

\numberwithin{equation}{section}

\begin{document}

\title[Morse Theory]{Morse Theory on 1-Skeleta}
\author{Chris McDaniel}
\address{Dept. of Math. and Comp. Sci.\\
Endicott College\\
Beverly, MA 01915}
\email{cmcdanie@endicott.edu}






\begin{abstract} 
Guillemin and Zara gave necessary and sufficient conditions under which Morse theoretic techniques could be used to construct an additive basis for the equivariant cohomology of a 1-skeleton that is either $3$-independent or GKM.  We show that their conditions remain valid for all 1-skeleta, $3$-independent, GKM, or otherwise.

\end{abstract}
\maketitle



\section{Introduction}
Let $\Gamma$ be a $d$-valent graph with vertex set $V_\Gamma$ and oriented edge set $E_\Gamma$ (i.e. $pq\in E_\Gamma \Leftrightarrow qp\in E_\Gamma$).  An \emph{axial function} on $\Gamma$ is a function $\alpha\colon E_\Gamma\rightarrow\R^n$ which maps oppositely oriented edges to opposite vectors, maps oriented edges issuing from each single vertex to pairwise linearly independent vectors, and satisfies the following coplanarity condition:  for each oriented edge $pq\in E_\Gamma$ and for any other oriented edge $e$ issuing from $p$ there is a corresponding oriented edge $\theta_{pq}(e)$ issuing from $q$ and a positive scalar $\lambda_{pq}(e)$ such that the difference $\alpha(e)-\lambda_{pq}(e)\alpha(\theta_{pq}(e))$ is collinear with $\alpha(pq)$.  The collection of oriented edge matchings $\theta_{pq}\colon E^p\rightarrow E^q$ ($E^x=$ the oriented edges issuing from vertex $x$) and the collection of positive scalar functions $\lambda_{pq}\colon E^p\rightarrow \R_+$ are called a \emph{connection} ($\theta$) and a \emph{compatibility system} ($\lambda$) for the pair $(\Gamma,\alpha)$, respectively.  The quadruple $(\Gamma,\alpha,\theta,\lambda)$ is called a $d$-valent \emph{1-skeleton} in $\R^n$.  

The \emph{equivariant cohomology} of the 1-skeleton $(\Gamma,\alpha,\theta,\lambda)$ is the set $H(\Gamma,\alpha)$ consisting of maps $f\colon V_\Gamma\rightarrow S\coloneqq\Sym(\R^n)$ such that for any $pq\in E_\Gamma$ the difference $f(q)-f(p)$ is in the ideal of $S$ generated by the linear element $\alpha(pq)$.  Vertex-wise addition and multiplication give $H(\Gamma,\alpha)$ the structure of a graded algebra over the polynomial ring $S$.  In two beautiful papers \cite{GZ1,GZ2}, Guillemin and Zara showed how Morse theoretic techniques could be used to construct a nice $S$-module basis for $H(\Gamma,\alpha)$ called a \emph{generating family} for certain 1-skeleta satisfying something called the \emph{acyclicity axiom}.  Following Guillemin and Zara we say that a 1-skeleton satisfying the acyclicity axiom has the \emph{Morse package} if it admits a generating family.  It turns out that the Morse package for a 1-skeleton is equivalent to the Morse package for its planar subskeleta, called \emph{$2$-slices}.  This is our main result:
\begin{theorem}
\label{thm:GZMain}
Assume $(\Gamma,\alpha,\theta,\lambda)$ satisfies the acyclicity axiom.  Then $(\Gamma,\alpha,\theta,\lambda)$ has the Morse package if and only if every $2$-slice has the Morse package.
\end{theorem}

In their important paper \cite{GZ1}, Guillemin and Zara essentially proved Theorem \ref{thm:GZMain} for 1-skeleta that are \emph{$3$-independent}, i.e. for each $p\in V_\Gamma$ and for any three oriented edges $e_1,e_2,e_3$ issuing from $p$ the vectors $\alpha(e_1), \alpha(e_2), \alpha(e_3)$ are linearly independent.  In a subsequent paper \cite{GZ2}, Guillemin and Zara proved Theorem \ref{thm:GZMain} for 1-skeleta satisfying the so-called \emph{GKM condition}, i.e. each of the scalar functions $\lambda_{pq}$ is identically equal to one.  The proof given in their latter paper is subtle and clever, drawing on some localization results they had obtained in an earlier paper \cite{GZ0}.  In this paper we show that the ideas developed by Guillemin and Zara remain valid in the general case, and we use these ideas to give a uniform proof of Theorem \ref{thm:GZMain} for all 1-skeleta, GKM, $3$-independent, or otherwise.  Before getting into the details of our proof we give a bit of background which should justify some of topological jargon used here, and give some motivation for studying these objects.

1-skeleta were first explicitly defined as above by Guillemin and Zara \cite{GZ1} as a combinatorial tool to study \emph{GKM manifolds}, after the seminal work of Goresky-Kottwitz-MacPheson \cite{GKM}.  A GKM manifold is a compact $2d$-dimensional almost complex manifold $M$ with a $T=\left(S^1\right)^n$-action whose zero and one (complex) dimensional orbits consist of finitely many $T$-invariant $S^2$'s, each containing exactly two fixed points.  An axial function for this ``graph'' is then defined by the weights of the isotropy representations of $T$ on the tangent spaces of $M$ at the fixed points, a connection is defined by a $T$-equivariant connection on the tangent bundle of $M$, and smoothness of $M$ guarantees that the constants $\lambda_{pq}(e)$ are always equal to one.  Hence a $2d$-dimensional GKM manifold with an $n$-dimensional torus acting defines a $d$-valent GKM 1-skeleton in $\Lie(T)^*\cong\R^n$.  A remarkable result of Goresky, Kottwitz, and MacPherson then states that if the manifold $M$ satisfies a technical condition called equivariant formality, then its $T$-equivariant cohomology $H_T(M)$ is isomorphic to the equivariant cohomology of its 1-skeleton.  One particularly nice family of equivariantly formal GKM manifolds are the \emph{Hamiltonian GKM manifolds}, i.e. symplectic GKM manifolds equipped with Hamiltonian torus actions.  In a series of three papers \cite{GZ0,GZ1,GZ2}, Guillemin and Zara showed that several topological results on Hamiltonian GKM manifolds have nice combinatorial interpretations on 1-skeleta, even when there is no underlying GKM manifold to speak of.   

For example, let $M$ be a Hamiltonian GKM manifold with torus $T$ acting.  Paraphrasing Guillemin and Zara \cite{GZ1}:  Let $\xi\in\Lie(T)$ be a generic covector, so that the fixed point set of its circle subgroup $H\coloneqq\langle\exp(t\xi)\rangle\subset T$ coincides with that of $T$, and let $\phi\colon M\rightarrow\R$ be its associated Hamiltonian function.  Then $\phi$ is a Morse-Bott function on $M$ whose critical points are exactly the $T$-fixed points $M^T\subset M$.  The equivariant Thom classes of the ``flow-up'' (or unstable) manifolds $\left\{W_p\right\}_{p\in V_\Gamma}$ relative the gradient flow of $\phi$ form a basis for $H_T(M)$ as a module over $H_T(pt)\cong \Sym(\Lie(T)^*)$, called a \emph{generating family}.  For each regular value $c\in\R\setminus \phi(M^T)$ the reduced space $M_c\coloneqq \phi^{-1}(c)/H$ is a Hamiltonian orbifold.  The flip-flop theorem then says that for two regular values separated by a single critical value, say $c<\phi(p)<c'$, the two reduced spaces $M_c$ and $M_{c'}$ are related by a blow up/blow down procedure called a \emph{flip-flop}.  The flip-flop theorem allows one to relate the $T'\coloneqq T/H$-equivariant cohomologies of the reduced spaces of $M$ to each other.  Furthermore the $T$-equivariant cohomology of $M$ is related to the $T'$-equivariant cohomology of its reduced space via the \emph{Kirwan map} $\Kirc\colon H_T(M)\rightarrow H_{T'}(M_c)$, and a theorem of Kirwan says that this map must be surjective. 

By analogy, a 1-skeleton satisfying the acyclicity axiom with fixed covector $\xi\in\left(\R^n\right)^*$ always admits a compatible \emph{Morse function} $\phi\colon V_\Gamma\rightarrow\R$.  A \emph{generating family}, if one exists, is a homogeneous $S$-module basis for $H(\Gamma,\alpha)$ given by Thom classes on the ``flow up'' subgraphs $\left\{\mathcal{F}_p\right\}_{p\in V_\Gamma}$ relative to the acyclic orientation on $\Gamma$ induced by $\xi$.  For any \emph{regular value}, $c\in\R\setminus\phi(V_\Gamma)$ define the \emph{cross section} of $(\Gamma,\alpha,\theta,\lambda)$ at level $c$ to be the pair $\Gamma_c\coloneqq(V_c,E_c)$ where $V_c$ is the set of oriented edges of $\Gamma$ that cross $c$ level, and $E_c$ are the $2$-slices that cross $c$ level.  Intuitively, one should think of the $c$-cross section as the intersection of the 1-skeleton with the $c$-translate of the annihilator hyperplane of $\xi$, $W_\xi\subset\R^n$.  In the case that $(\Gamma,\alpha,\theta,\lambda)$ is $3$-indepdendent with $2$-slices that have the Morse package, it turns out that $\Gamma_c$ is a $(d-1)$-valent graph which inherits a 1-skeleton structure from $(\Gamma,\alpha,\theta,\lambda)$.  In this case, one can then show that these cross sectional 1-skeleta satisfy an analogue of the flip-flop theorem.  Moreover, there is an analogue of the Kirwan map $\Kirc\colon H(\Gamma,\alpha)\rightarrow \Maps(V_c,S_\xi)$ ($S_\xi=\Sym(W_\xi)$) and, again in the case that $(\Gamma,\alpha,\theta,\lambda)$ is $3$-independent with Morse $2$-slices, one can show that its image $\Kirc\left(H(\Gamma,\alpha)\right)$ is equal to $H(\Gamma_c)$, the equivariant cohomology of $\Gamma_c$ with its inherited 1-skeleton structure.  Using these facts, Guillemin and Zara \cite{GZ1} were able to deduce that a $3$-independent 1-skeleton with Morse $2$-slices must have the Morse package itself, which is the hard implication in Theorem \ref{thm:GZMain}.  Without $3$-independence however, the cross section $\Gamma_c$ is not a 1-skeleton, and the cross sectional equivariant cohomology as defined above is not so well behaved.  Hence to carry our analogy over to the general case, one must find a suitable replacement for $H(\Gamma_c)$.  In their remarkable paper \cite{GZ2}, Guillemin and Zara found the ``right'' replacement for $H(\Gamma_c)$ using analogues of integral operators they had introduced and studied for GKM 1-skeleta in an earlier paper \cite{GZ0}. 

In that paper \cite{GZ0}, Guillemin and Zara used an analogue of the Atiyah-Bott-Berline-Vergne localization formula for the integral of an equivariant cohomology class on a Hamiltonian GKM space to define an ``integral operator'' on the equivariant cohomology of a GKM 1-skeleton.  They also introduced the notion of residues to prove a combinatorial analogue of the Jeffrey-Kirwan theorem for GKM 1-skeleta, which gives rise to a ``cross sectional integral operator''.  It turns out that the existence of such integral operators is equivalent to the compatibility system of the 1-skeleton satisfying a kind of ``trivial holonomy'' condition that we call \emph{straightness}.  In particular for a straight 1-skeleton $(\Gamma,\alpha,\theta,\lambda)$ we show that there exist positive constants $\left\{c_p\right\}_{p\in V_\Gamma}$ such that for every $f\in H(\Gamma,\alpha)$ the sum of rational functions
\begin{equation}
\label{eq:integralintro}
\int_\Gamma f\coloneqq \sum_{p\in V_\Gamma}\frac{f(p)}{c_p\prod_{e\in E^p}\alpha(e)}
\end{equation}
is actually a polynomial in $S$, c.f. Proposition \ref{prop:StrInt}.  The map $\int_\Gamma\colon H(\Gamma,\alpha)\rightarrow S[d]$ is called an \emph{integral operator} on $H(\Gamma,\alpha)$.  One of the important observations made by Guillemin and Zara \cite{GZ2} is that a map $h\colon V_\Gamma\rightarrow S$ is in $H(\Gamma,\alpha)$ if and only if 
$$\int_\Gamma f\cdot h \in S \ \ \forall \ f\in H(\Gamma,\alpha).$$

Now assume that $(\Gamma,\alpha,\theta,\lambda)$ satisfies the acyclicity axiom, with $\xi\in\left(\R^n\right)^*$, $\phi\colon V_\Gamma\rightarrow\R$, and $c\in\R\setminus\phi(V_\Gamma)$ fixed, and for each $f\in\Maps(V_c,S_\xi)$ define:
$$\int_{\Gamma_c} f\coloneqq\sum_{e\in V_c}\frac{f(e)}{c_{i(e)}m_e\prod_{\substack{e'\in E^{i(e)}\\ e'\neq e\\}}\rho_e(\alpha(e'))}$$
where $i(e)$ is the initial vertex of $e$, $c_{i(e)}$ is as in \eqref{eq:integralintro}, $m_e=\langle\xi,\alpha(e)\rangle$, and $\rho_e\colon\R^n\rightarrow W_\xi$ is the projection along the $\alpha(e)$ coordinate.  Under these assumptions (i.e. straightness and acyclicity), we prove that for every $h\in H(\Gamma,\alpha)$ we have 
\begin{equation}
\label{eq:integralcintro}
\int_{\Gamma_c}\Kirc(h)=\sum_{\phi(q)<c}\frac{1}{c_q}\Res_\xi\left(\frac{f(p)}{\prod_{e\in E^p}\alpha(e)}\right)
\end{equation}
where $\Res_\xi$ is the residue operator on rational functions introduced by Guillemin and Zara \cite{GZ0}, c.f. Lemma \ref{prop:KirwanIntegral}.  In particular this implies that $\int_{\Gamma_c}\Kirc(h)$ is a polynomial in $S_\xi$ for every $h\in H(\Gamma,\alpha)$.  

Following Guillemin and Zara, one now defines the \emph{cross sectional equivariant cohomology} $H(\Gamma_c)$ to be the set of maps $f\in\Maps(V_c,S_\xi)$ such that 
$$\int_{\Gamma_c}f\cdot\Kirc(h) \in S_\xi \ \ \forall \ h\in H(\Gamma,\alpha).$$
While this definition only makes sense for 1-skeleta that are straight, we show that in fact straightness for $(\Gamma,\alpha,\theta,\lambda)$ is implied by either the Morse package on $(\Gamma,\alpha,\theta,\lambda)$, or the Morse package on its $2$-slices, c.f. Propositions \ref{prop:StMor} and \ref{prop:StMor2}.

The remarkable fact is that an analogue of the flip-flop theorem actually holds for $H(\Gamma_c)$ as defined above.  More precisely, we show that if $(\Gamma,\alpha,\theta,\lambda)$ satisfies the acyclicity axiom and if every $2$-slice of $(\Gamma,\alpha,\theta,\lambda)$ has the Morse package, then for any two regular values separated by a unique critical value, $c<\phi(p)<c'$, there are $S_\xi$-module maps 
\begin{equation}
\label{eq:FlipFlopMaps}
\xymap{H(\Gamma_c)\ar@/^/[r]^-{\mu_p} & H(\Gamma_{c'})\ar@/^/[l]^-{\delta_p}},
\end{equation}
called \emph{flip-flop maps}, c.f. Lemma \ref{lem:Flip}.  Using the flip-flop maps, we prove that the image of the Kirwan map is again equal to $H(\Gamma_c)$, c.f. Proposition \ref{prop:SurjKirwan}.  It turns out that one can then deduce the existence of a generating family for $(\Gamma,\alpha,\theta,\lambda)$ directly from the surjectivity of the Kirwan maps, c.f. Lemma \ref{lem:SurjKMorse}.  

We would like to add that several of the results in this paper were proved for GKM 1-skeleta by Guillemin and Zara, and, once straightness has been established, many of their proofs actually hold verbatim.  Several of these proofs have been reproduced here for the sake of completeness, with references to the original arguments of Guillemin and Zara.  On the other hand there are some new arguments in this paper that give our proof of Theorem \ref{thm:GZMain} a slightly distinctive flavor from that given by Guillemin and Zara, for better or worse.     

This paper is organized as follows.  In Section \ref{sec:Basic} we define the fundamental notions of a 1-skeleton and its equivariant cohomology, and establish some of their properties.  In Section \ref{sec:Proof} we prove Theorem \ref{thm:GZMain}.  In Section \ref{sec:Comments} we give a few concluding remarks regarding 1-skeleta in the plane.   

Unless otherwise stated:  All rings and modules are graded over $\N=\left\{0,1,2,\ldots,\right\}$, all maps are assumed to be graded of degree zero, and $M[i]$ denotes the shifted module $M[i]^j\coloneqq M^{i+j}$.

\section{Definitions}
\label{sec:Basic}
\subsection{1-Skeleta}
A graph $\Gamma$ is a pair consisting of vertices $V_\Gamma$ and oriented edges $E_\Gamma$ by which we mean distinct ordered pairs of vertices where $pq\in E_\Gamma$ if and only if $qp\in E_\Gamma$.  For a given oriented edge $e=pq\in E_\Gamma$, its \emph{initial} vertex is $i(e)\coloneqq p$, its \emph{terminal} vertex is $t(e)\coloneqq q$, and its oppositely oriented counterpart is $\bar{e}\coloneqq qp$.  By the set of oriented edges at $p$, denoted $E^p$, we mean the set of all oriented edges with initial vertex $p$.  The graph is $d$-valent if the cardinality of $E^p$ is equal to $d$ for every $p\in V_\Gamma$.  Unless otherwise stated, all graphs in this paper are connected and have constant valency.  

A \emph{connection} $\theta$ on $\Gamma$ is a collection of bijective maps $\theta_{pq}\colon E^p\rightarrow E^q$ indexed by the set $E_\Gamma$ satisfying $\theta_{pq}(pq)=qp$ and $\theta_{pq}=\theta_{qp}^{-1}$ for each $pq\in E_\Gamma$.  An \emph{axial function} $\alpha$ on $\Gamma$ compatible with $\theta$ is a map $\alpha\colon E_\Gamma\rightarrow\R^n$ satisfying the following axioms:
\begin{enumerate}[{A}1.]
\item ${\dsp \left\{\alpha(e)\left|\right.e\in E^p\right\}}$ is pairwise linearly independent for each $p\in V_\Gamma$
\item $\alpha(pq)=-\alpha(qp)$ for each $pq\in E_\Gamma$
\item For every vertex $p\in V_\Gamma$ and for each pair $e,e'\in E^p$ there exist positive constants $\lambda_e(e')$ such that
$$\alpha(e')-\lambda_e(e')\alpha(\theta_e(e'))\equiv 0 \ \ \ \text{mod} \ \alpha(e)$$
\end{enumerate}
It is convenient to regard the positive constants as a family of functions $\lambda\coloneqq\left\{\lambda_e\colon E^{i(e)}\rightarrow\R_+\right\}_{e\in E_\Gamma}$ called the \emph{compatibility system} for the triple $(\Gamma,\alpha,\theta)$.  We define $\lambda_e(e)\coloneqq 1$.  Note that $\lambda$ is uniquely determined by the triple $(\Gamma,\alpha,\theta)$.  

\begin{definition}
\label{def:1skeleton}
A $d$-valent 1-skeleton in $\R^n$ is a quadruple $(\Gamma,\alpha,\theta,\lambda)$ consisting of a $d$-valent graph $\Gamma$, a connection $\theta$ on $\Gamma$, a compatibility system $\lambda$ on $\Gamma$ compatible with $\theta$, and an axial function $\alpha$ on $\Gamma$ compatible with $\theta$ and $\lambda$.
\end{definition}

A 1-skeleton is called \emph{GKM} if its compatibility system satisfies $\lambda_e\equiv 1$ for each $e\in E_\Gamma$.  A 1-skeleton is called \emph{$k$-independent} if for each $p\in V_\Gamma$ every $k$-subset of vectors $\left\{\alpha(e)\left|\right. e\in E^p\right\}$ is linearly independent.  Note that a 1-skeleton is always $2$-independent.

\subsection{Subskeleta}
 Let $\Gamma_0\subseteq\Gamma$ be a subgraph.  Suppose that for each $pq\in E_0$, the function $\theta_{pq}\colon E^p\rightarrow E^q$ restricts to a function on $\theta_0\colon E^p_0\rightarrow E^q_0$.  Then $\Gamma_0$ inherits the 1-skeleton structure from $\Gamma$ and we call the 1-skeleton $(\Gamma_0,\alpha_0,\theta_0,\lambda_0)$ a subskeleton.  The \emph{normal edges} to $\Gamma_0$ at $p\in V_0$ are the edges at $p$ that are not in $\Gamma_0$, i.e. $N^p_0\coloneqq E^p\setminus E^p_0$.  Note that $\theta_{pq}$ also defines \emph{normal connection maps} $\theta_{pq}^\perp\colon N_0^p\rightarrow N_0^q$.

For any $k$-dimensional subspace $H\subseteq\R^n$, let $\Gamma_H$ be the disjoint union of subgraphs of $\Gamma$ whose oriented edge set is $E_H\coloneqq\left\{e\in E_\Gamma\left|\right. \alpha(e)\in H\right\}$.  Let $\Gamma^0_H$ be any connected component of $\Gamma_H$.  Then $\Gamma^0_H$ has constant valency.  Indeed if $pq\in E^0_H$, and $e\in E_{H,p}^0$, then we have $\alpha(e)-\lambda_{pq}(e)\alpha(\theta_{pq}(e))=c_{pq}\alpha(pq)$.  Since $\alpha(e), \alpha(pq)\in H$ we conclude that $\alpha(\theta(e))\in H$ as well.  Thus in particular we have that $\theta_{pq}\left(E^0_{H,p}\right)=E^0_{H,q}$ for every $pq\in E^0_H$.  Since $\Gamma_H^0$ is connected we get a subskeleton $(\Gamma_H^0,\alpha_H^0,\theta_H^0,\lambda_H^0)$ called a \emph{$k$-slice} with respect to $H\subseteq\R^n$.  Note that a $k$-slice need not be $k$-valent.

\subsection{Paths and Holonomy}
For every path (resp. loop) $\gamma\colon p_0\sra p_1\cdots\sra p_N$ in $\Gamma$ composing the connection maps along the edges of $\gamma$ yields the \emph{path connection map} (resp. \emph{holonomy map}) for $\gamma$, 
$$K_\gamma\coloneqq\theta_{p_{N-1}p_N}\circ\theta_{p_0p_1}\colon E^{p_0}\rightarrow E^{p_N}.$$ 
Similarly, the product of the compatibility maps along the edges of $\gamma$ yields the \emph{path connection number} (resp. \emph{holonomy number}) for $\gamma$, 
$$\left|K_\gamma\right|\coloneqq\prod_{i=0}^{N-1}\prod_{e\in E^{p_i}}\lambda_{p_ip_{i+1}}(e).$$  
For example if $\gamma\colon p\sra q$ consists of a single edge, then $\left|K_{pq}\right|=\prod_{e\in E^p}\lambda_{pq}(e)$.  For each path $\gamma\colon p\sra\cdots\sra q$, let $\bar{\gamma}\colon q\sra\cdots\sra p$ denote the path $\gamma$ traversed backwards.  For two paths $\gamma''\colon p\sra\cdots\sra q$ and $\gamma'\colon q\sra\cdots\sra r$, their \emph{compostion} $\gamma'\circ\gamma''\colon p\sra\cdots q\sra\cdots \sra r$ is defined by concatenation.  Clearly we have $K_{\gamma'\circ\gamma''}=K_{\gamma'}\circ K_{\gamma''}$, where the $\circ$ on the right is function composition.  Also note that $\left|K_{\gamma'\circ\gamma''}\right|=\left|K_{\gamma'}\right|\cdot\left|K_{\gamma''}\right|$.  In particular, for any path $\gamma\colon p\sra\cdots\sra q$ we always have $K_{\gamma}\circ K_{\bar{\gamma}}=id_{E_p}$ and $\left|K_{\gamma}\right|\cdot\left|K_{\bar{\gamma}}\right|=1$.  A \emph{loop with basepoint $p$} is any path of the form $\gamma\colon p\sra\cdots\sra p$.  Note that if $\gamma\colon q\sra\cdots \sra q$ is any loop with basepoint $q$, and $\gamma'\colon p\sra\cdots\sra q$ is any path from $p$ to $q$, then $\gamma''\coloneqq \bar{\gamma'}\circ\gamma\circ\gamma'$ is a loop with basepoint $p$ and $\left|K_{\gamma''}\right|=\left|K_\gamma\right|$.  
\begin{definition}
\label{def:Straight}
A 1-skeleton $(\Gamma,\alpha,\theta,\lambda)$ is \emph{straight} if $\left|K_\gamma\right|=1$ for every loop $\gamma$ in $\Gamma$.
\end{definition}
Note that in order to verify straightness, it suffices to just check those loops with a fixed basepoint.

  If the path $\gamma$ lies in a subskeleton $\Gamma_0$, then we can restrict the holonomy maps to the normal edges to $\Gamma_0$ to define the \emph{normal holonomy maps} and the \emph{normal holonomy numbers}.  We then say that the subskeleton is \emph{normally straight} if the normal path connection numbers $\left|K_\gamma^\perp\right|$ are equal to one for evey loop $\gamma$ in $\Gamma_0$. 
Note that a subskeleton of a straight 1-skeleton is straight if and only if it is normally straight.  Indeed if $\gamma$ is a loop in $\Gamma_0$, then the path connection number factors $\left|K_\gamma\right|=\left|K_\gamma^0\right|\cdot\left|K_\gamma^\perp\right|$.

\begin{proposition}
\label{prop:ksliceNormStr}
Every $k$-slice is normally straight.
\end{proposition}
\begin{proof}
Fix a $k$-subspace $H\subseteq\R^n$, and let $(\Gamma_H^0,\alpha_H^0,\theta_H^0,\lambda_H^0)$ be a $k$-slice.  Choose a covector $\eta\in\left(\R^n\right)^*$ which vanishes on $H$, but does not vanish on $\alpha(e)$ for each $e\notin N_0$.  Let $K_\eta$ be the $(n-1)$-dimensional annihilator subspace of $\eta\colon\R^n\rightarrow\R$, and let $P_{e,e'}$ denote the $2$-dimensional subspace spanned by $\alpha(e)$, $\alpha(e')$, and $\alpha\left(\theta_e(e')\right)$.  Note that by our choice of $\eta$, $\dim\left(P_{e,e'}\cap K_\eta\right)=1$.  Now for each $e\in E_0^p$ and $e'\in N_0^p$ we have 
$$\alpha(e')-\frac{\langle\eta,\alpha(e')\rangle}{\langle\eta,\alpha(\theta_e(e'))\rangle}\alpha(\theta_e(e'))\in P_{e,e'}\cap K_\eta.$$  
On the other hand we also have 
$$\alpha(e')-\lambda_e(e')\alpha(\theta_e(e'))=c_e\alpha(e)\in P_{e,e'}\cap K_\eta.$$
Since the subspace $P_{e,e'}\cap K_\eta$ is $1$-dimensional, it must be spanned by $\alpha(e)$.  It follows that 
$$\alpha(e')-\frac{\langle\eta,\alpha(e')\rangle}{\langle\eta,\alpha(\theta_e(e'))\rangle}\alpha(\theta_e(e'))\in \R\cdot\alpha(e).$$
Moreover, since $\alpha(e)=-\alpha(\bar{e})$ and $\alpha\left(\theta_e(e')\right)$ are linearly independent it follows that 
\begin{equation}
\label{eq:HolkSlice1}
\lambda_e(e')=\frac{\langle\eta,\alpha(e')\rangle}{\langle\eta,\alpha(\theta_e(e'))\rangle}.
\end{equation}
Now for any loop $\gamma\colon p_0\sra p_1\sra\cdots\sra p_N\sra p_0$ in $\Gamma^0_H$, \eqref{eq:HolkSlice1} yields 
\begin{equation}
\label{eq:HolkSlice2}
\left|K_\gamma^\perp\right|= \prod_{i=0}^N\left(\prod_{e\in N^{p_i}_0}\frac{\langle\eta,\alpha(e)\rangle}{\langle\eta,\alpha(\theta_{p_ip_{i+1}}(e))\rangle}\right)=\frac{\prod_{i=0}^N\prod_{e\in N^{p_i}_0} \langle\eta,\alpha(e)\rangle}{\prod_{i=0}^N\prod_{e\in N^{p_{i+1}}_0}\langle\eta,\alpha(\theta_{p_ip_{i+1}}(e))\rangle}.
\end{equation}
Since the same factors occur in both the numerator and the denominator, the quotient in \eqref{eq:HolkSlice2} must equal $1$, hence the $k$-slice is normally straight.
\end{proof}

\subsection{Polarizations}
A covector $\xi\in\left(\R^n\right)^*$ with the property that $\langle\alpha(e),\xi\rangle$ is called \emph{polarizing}.  Note that any polarizing covector induces an orientation on $\Gamma$ by specifying for each pair of oriented edges $e,\bar{e}\in E_\Gamma$ the one for which $\langle\alpha(e),\xi\rangle >0$.  If the orientation on $\Gamma$ induced by $\xi$ is acyclic, i.e. no oriented loops, then we say that $\Gamma$ is \emph{$\xi$-acyclic}.  Note that a $\xi$-acyclic orientation gives a partial ordering to the vertex set $V_\Gamma$ defined by $p\leq q$ if and only if there is an oriented path in $\Gamma$ from $p$ to $q$.  As Guillemin and Zara point out \cite[Theorem 1.4.1]{GZ1}, this partial order on $V_\Gamma$ extends naturally to a total order as follows.  For each vertex $p\in V_\Gamma$ define $\hat{\phi}(p)$ to be the length of the longest oriented path in $\Gamma$ which terminates at $p$.  Now perturb the values of $\hat{\phi}$ slightly to obtain an injective function $\phi\colon V_\Gamma\rightarrow \R$ with the property that $\phi(p)<\phi(q)$ whenever $\langle\alpha(pq),\xi\rangle>0$.  Such a function $\phi$ is called a ($\xi$-compatible) \emph{Morse function} for $\Gamma$.  Clearly the existence of a Morse function depends on the existence of a $\xi$-acyclic orientation.  It turns out that there are some examples of 1-skeleta which do not admit any $\xi$-acyclic orientations at all, c.f. \cite[pg. 302]{GZ1} or \cite[pg. 950]{McD2}.  To use Morse theory on 1-skeleta we must therefore invoke the so-called acyclicity axiom:

\noindent\textbf{Acyclicity Axiom:}  There exists a polarizing covector $\xi\in\left(\R^n\right)^*$ for $(\Gamma,\alpha,\theta,\lambda)$ which induces a $\xi$-acyclic orientation on $\Gamma$.

For technical reasons we would also like our polarizing covectors to be \emph{generic}, meaning that for each $p\in V_\Gamma$ and each quadruple $e_1,e_2,e_3,e_4\in E^p$ we have 
$$\frac{\alpha(e_1)}{\langle\xi,\alpha(e_1)\rangle}-\frac{\alpha(e_2)}{\langle\xi,\alpha(e_2)\rangle}\neq\frac{\alpha(e_3)}{\langle\xi,\alpha(e_3)\rangle}-\frac{\alpha(e_4)}{\langle\xi,\alpha(e_4)\rangle}.$$
Note that the set of generic polarizing covectors is a non-empty Zariski open set.  Thus if $(\Gamma,\alpha,\theta,\lambda)$ satisfies the acyclicity axiom, then it is always possible to find a generic polarizing covector $\xi\in\left(\R^n\right)^*$ whose induced orientation on $\Gamma$ is acyclic.  

\emph{From now on, we will assume that the acyclicity axiom is satisfied by our 1-skeleta, and a generic polarizing covector for a 1-skeleton will always refer to one whose induced orientation is acyclic.}

Fix a generic polarizing covector $\xi\in\left(\R^n\right)^*$ for $(\Gamma,\alpha,\theta,\lambda)$.  The set of oriented edges at a vertex $p\in V_\Gamma$ whose directions pair positively with $\xi$ are said to \emph{flow out from $p$}, denoted $E^p_+$, and those whose directions pair negatively with $\xi$ are said to \emph{flow into $p$}, denoted $E^p_-$.  The cardinality of the set $E^p_-$ is called the index of $p$ with respect to $\xi$, denoted by $\ind_\xi(p)$.  

\begin{definition}
For $0\leq i\leq d$, the $i^{th}$ combinatorial Betti number for $(\Gamma,\alpha,\theta,\lambda)$ is the number
$$b_i(\Gamma,\alpha)\coloneqq\#\left\{p\in V_\Gamma\left| \ind_\xi(p)=i\right.\right\}.$$
\end{definition}
While the index of a vertex clearly depends on the choice of $\xi$, the remarkable fact is that the combinatorial Betti numbers do not depend on $\xi$, c.f. \cite[Theorem 1.3.1]{GZ1}.  From this we observe that the combinatorial Betti numbers are symmetric, i.e.  for $0\leq i\leq d$ $b_i(\Gamma,\alpha)=b_{d-i}(\Gamma,\alpha)$.  Indeed, if $\xi$ is a generic polarizing covector then so is $-\xi$, and we have for each $p\in V_\Gamma$ $\ind_{\xi}(p)=d-\ind_{-\xi}(p)$.  

Also note that from the acyclicity, we must have $b_0(\Gamma,\alpha)\geq 1$ and hence $b_d(\Gamma,\alpha)\geq 1$.  A 1-skeleton with $b_0(\Gamma,\alpha)=1=b_d(\Gamma,\alpha)$ is called \emph{pointed}.  In this case there is a unique \emph{source vertex} $p_{0}\in V_\Gamma$ whose edges all flow out, and a unique \emph{sink vertex} $p_1\in V_\Gamma$ whose edges all flow in.  A 1-skeleton is called \emph{non-cyclic} if every $2$-slice pointed.  It is not difficult to see that a non-cyclic 1-skeleton is always pointed.  The converse is not true however, c.f. \cite[pg. 951]{McD2}.  The following result is key to further results in this paper.     
\begin{proposition}
\label{prop:Straightness}
Let $(\Gamma,\alpha,\theta,\lambda)$ be any $d$-valent non-cyclic 1-skeleton in $\R^n$.  Then $(\Gamma,\alpha,\theta,\lambda)$ is straight if and only if every $2$-slice $(\Gamma_H^0,\alpha_H^0,\theta_H^0,\lambda_H^0)$ is straight.
\end{proposition}
\begin{proof}
If $(\Gamma,\alpha,\theta,\lambda)$ is straight, then it follows from Proposition \ref{prop:ksliceNormStr} that every $2$-slice (in fact every $k$-slice) is straight.  Now assume that every $2$-slice is straight.  Fix a generic polarizing covector, let $\phi\colon V_\Gamma\rightarrow\R$ be a compatible Morse function, and let $p_0$ be the minimum vertex with respect to $\phi$.  For each loop $\gamma$ based at $p_0$ let $h(\gamma)$ be the largest vertex of $\gamma$ with respect to $\phi$, and let $\mu(\gamma)$ be the number of times that $\gamma$ passes through it.  Then define the \emph{height} of $\gamma$ to be the pair $\left(h(\gamma),\mu(\gamma)\right)\in V_\Gamma\times\Z_+$.  We endow the set $V_\Gamma\times\Z_{+}$ with the lexicographic ordering (i.e. $(p,n)\leq(q,m)$ if and only if either $p<q$ or $p=q$ and $n<m$).  We will prove that $|K_\gamma|=1$ by induction on the height of $\gamma$.

First assume the height of $\gamma$ is $(p_0,1)$.  Then $\gamma$ is the trivial loop and there is nothing to show.  Now assume that $\gamma$ is a loop based at $p_0$ with height $(x,m)$.  Then we can write $\gamma\colon p_0\sra p_1\sra\cdots\sra p_{m-1}\sra x\sra p_{m+1}\sra\cdots\sra p_n\sra p_0$, where $\phi(p_i)\leq\phi(x)$ for all $i$.  Define $\gamma_0\colon p_0\sra p_1\sra\cdots \sra p_{m-1}$, $\gamma_{2}\colon p_{m+1}\sra\cdots p_n\sra p_0$, and $\gamma_1\colon p_{m-1}\sra x\sra p_{m+1}$, so that $\gamma=\gamma_2\circ\gamma_1\circ\gamma_0$.  Let $\Gamma_H^0=(V_0,E_0)$ be the graph of the $2$-slice containing $x$ spanned by $\alpha(xp_{m-1})$ and $\alpha(xp_{m+1})$, and let $s\in V_0$ denote the unique source in $\Gamma_H^0$, which exists by our non-cyclicity assumption.  By uniqueness of $s$, there exist oriented paths in $\Gamma_H^0$, say $\gamma_4\colon s\sra t_1\sra\cdots\sra t_k=p_{m-1}$ and $\gamma_3\colon s\sra u_1\sra\cdots\sra u_\ell=p_{m+1}$.  Then $\bar{\gamma}_3\circ\gamma_{1}\circ\gamma_4\colon s\sra t_1\sra\cdots\sra p_{m-1}\sra x\sra p_{m+1}\sra\cdots\sra u_1\sra s$ is a loop in $\Gamma_H^0$ (based at $s$).  Note that $\phi(t_a),\phi(u_b)<\phi(x)$ for all $a,b$ since the paths $\gamma_3$ and $\gamma_4$ are oriented.  Since we are assuming $(\Gamma_H^0,\alpha_H^0,\theta_H^0,\lambda_H^0)$ is straight, we deduce that $\left|K_{\bar{\gamma}_3}\circ K_{\gamma_1}\circ K_{\gamma_4}\right|=1$ from which we deduce that
$$\left|K_{\gamma_1}\right|=\frac{\left|K_{\gamma_3}\right|}{\left|K_{\gamma_4}\right|}.$$  
On the other hand the new loop (based at $p_0$) 
$$\hat{\gamma}\coloneqq\gamma_2\circ\gamma_3\circ\bar{\gamma}_4\circ\gamma_0\colon p_0\sra\cdots\sra p_{m-1}\sra t_{k-1}\sra\cdots s\sra u_1\sra\cdots \sra p_{m+1}\sra\cdots p_n\sra p_0$$
has the same holonomy number as $\gamma$, i.e.
$$\left|K_{\hat{\gamma}}\right|=\left|K_{\gamma_2}\right|\cdot\frac{\left|K_{\gamma_3}\right|}{\left|K_{\gamma_4}\right|}\cdot\left|K_{\gamma_0}\right|=\left|K_{\gamma}\right|.$$
Since $\hat{\gamma}$ must have smaller height than $\gamma$, we deduce by induction that $\left|K_{\hat{\gamma}}\right|=1$, and thus $\left|K_\gamma\right|$ must equal $1$ as well.
\end{proof}

\subsection{Cohomology Rings}
Let $S$ denote the symmetric algebra on $\R^n$, and let $Q$ denote its field of fractions.  Let $\Maps(V_\Gamma,S)$ denote the set of set maps $f\colon V_\Gamma\rightarrow S$.  The set $\Maps(V_\Gamma,S)$ has the structure of a graded ring, where addition and multiplication are defined point-wise.  Moreover the polynomial ring $S$ sits inside $\Maps(V_\Gamma,S)$ as the subring of constant maps, hence $\Maps(V_\Gamma,S)$ is actually an $S$-algebra.
The \emph{equivariant cohomology} of $(\Gamma,\alpha,\theta,\lambda)$ is defined as the subset
$$H(\Gamma,\alpha)\coloneqq\left\{f\colon V_\Gamma\rightarrow S\left|\right. f(q)-f(p)\in\alpha(pq)\cdot S, \ \ \ \forall pq\in E_\Gamma\right\}.$$
It is straightforward to see that $H(\Gamma,\alpha)$ is a graded $S$-subalgebra of $\Maps(V_\Gamma,S)$.  Since $S$ is Noetherian and $\Maps(V_\Gamma,S)$ is finitely generated, the subalgebra $H(\Gamma,\alpha)$ must also be finitely generated.  Hence the quotient 
$$\overline{H(\Gamma,\alpha)}\coloneqq\frac{H(\Gamma,\alpha)}{S^+\cdot H(\Gamma,\alpha)}$$
is a finite-dimensional graded $\R$-algebra, called the \emph{ordinary cohomology} of $(\Gamma,\alpha,\theta,\lambda)$.  

\subsubsection{Thom Classes and Integral Operators}
The \emph{support} of an equivariant class $f\in H(\Gamma,\alpha)$ is the set of vertices on which $f$ is non-zero, denoted by $\operatorname{supp}(f)$.  A non-zero equivariant class $\tau_0\in H(\Gamma,\alpha)$ of degree $k$ whose support lies in the vertex set of a subskeleton with co-valence $k$ is called a \emph{Thom class} for that subskeleton.  Note the support and degree restrictions tell us that the Thom class of a subskeleton $\Gamma_0\subseteq\Gamma$ must have the form
\begin{equation}
\label{eq:ThomClassDescription}
\tau_0(p)=\begin{cases} t_p\cdot\prod_{e\in N^p_0}\alpha(e) & \text{if $p\in V_0$}\\ 0 & \text{otherwise}\\ \end{cases}
\end{equation}
for some real numbers $\left\{t_p\right\}_{p\in V_0}$ satisfying the following condition: 
\begin{equation}
\label{eq:CocycleCondition}
t_q=t_p\cdot\prod_{e\in N^p_0}\lambda_{pq}(e) \ \ \ \forall \ pq\in E_0.
\end{equation}  
Conversely, if we can define numbers $\left\{t_p\right\}_{p\in V_0}$ satisfying \eqref{eq:CocycleCondition} then the subskeleton $(\Gamma_0,\alpha_0,\theta_0,\lambda_0)$ must support a Thom class.  This turns out to be related to the normal straightness of the subskeleton. 
\begin{proposition}
\label{prop:ThomNormal}
A subskeleton $(\Gamma_0,\alpha_0,\theta_0,\lambda_0)$ supports a Thom class if and only if it is normally straight.
\end{proposition}
\begin{proof}
Assume that the subskeleton is normally straight.  Fix a basepoint $p_0$ and set $t_{p_0}\coloneqq 1$.  Now for any other vertex $q\neq p_0$ of $\Gamma_0$, define a path $\gamma_q\colon p_0\sra\cdots\sra q$ and define $t_q\coloneqq\left|K_{\gamma_q}^\perp\right|$.  Then for any edge $pq\in E_0$ we get a loop (based at $p_0$) $\gamma\coloneqq\bar{\gamma_q}\circ \left\{p\sra q\right\}\circ\gamma_p\colon p_0\sra\cdots\sra p\sra q\sra\cdots\sra p_0$.  Therefore we have that
$$\left|K_{\gamma}^\perp\right|=\frac{\left|K_{\gamma_p}^\perp\right|\cdot\left|K_{pq}^\perp\right|}{\left|K^\perp_{\gamma_q}\right|}=\frac{t_p}{t_q}\left|K_{pq}^\perp\right|.$$  
Hence by straightness, we have 
$$t_p\left|K_{pq}^\perp\right|=t_q=t_p\prod_{e\in N^p_0}\lambda_{pq}(e).$$
Hence our numbers $\left\{\left|K_{\gamma_p}^\perp\right|\right\}_{p\in V_0}$ satisfy \eqref{eq:CocycleCondition} which implies that $\Gamma_0$ supports a Thom class.  

Conversely assume that $\Gamma_0$ supports a Thom class, with constants $\left\{t_p\right\}_{p\in V_0}$ satisfying \eqref{eq:CocycleCondition}.  Let $\gamma\colon p_0\sra p_1\sra p_2\sra\cdots\sra p_N\sra p_0$ be any loop in $\Gamma_0$.  Then we have that 
\begin{align*}
\left|K_{\gamma}^\perp\right| & =\left|K^\perp_{p_0p_1}\right|\cdot\left|K^\perp_{p_1p_2}\right|\cdots\left|K^\perp_{p_Np_0}\right|\\
& = \frac{t_{p_0}}{t_{p_1}}\cdot\frac{t_{p_1}}{t_{p_2}}\cdots\frac{t_{p_N}}{t_{p_0}}\\
& = 1,
\end{align*}
as desired.
\end{proof}

Note that vertices and edges always have Thom classes.  More generally, by Proposition \ref{prop:ksliceNormStr}, every $k$-slice has a Thom class.  
For any subskeleton the set-inclusion $(\Gamma_0,\alpha_0,\theta_0,\lambda_0)\subseteq(\Gamma,\alpha,\theta,\lambda)$ defines a restriction map $H(\Gamma,\alpha)\rightarrow H(\Gamma_0,\alpha_0)$.  If the subskeleton supports a Thom class then multiplication by that Thom class defines a map in the other direction.
\begin{proposition}
\label{prop:whatever}
Given a $k$-valent subskeleton $(\Gamma_0,\alpha_0,\theta_0,\lambda_0)\subseteq(\Gamma,\alpha,\theta,\lambda)$ with a Thom class $\tau_0\in H(\Gamma,\alpha)$, multiplication defines a map of $S$-modules
$$\xymap{H(\Gamma_0,\alpha_0)\ar[r]^-{\times \tau_0} & H(\Gamma,\alpha)[k]\\
f\ar@{|->}[r] & f\cdot\tau_0.\\}$$  
\end{proposition}
\begin{proof}
Let $(\Gamma_0,\alpha_0,\theta_0,\lambda_0)$ be a $k$-valent subskeleton of $(\Gamma,\alpha,\theta,\lambda)$, let $\tau_0\in H(\Gamma,\alpha)$ be a Thom class, and let $f\in H(\Gamma_0,\alpha_0)$ be an arbitrary element.  We need to check that for each $pq\in E_\Gamma$ we have 
\begin{equation}
\label{eq:whatever1}
f(q)\cdot\tau_0(q)-f(p)\cdot\tau_0(p)=c\cdot\alpha(pq)
\end{equation}
for some $c\in S$.  Clearly if neither one of $p$ or $q$ are vertices in $\Gamma_0$, then by our description of Thom classes in \eqref{eq:ThomClassDescription}, \eqref{eq:whatever1} is satisfied.  Suppose that both $p$ and $q$ are vertices in $\Gamma_0$.  If $pq$ is an oriented edge of $\Gamma_0$ then \eqref{eq:whatever1} must be satisfied since both $f$ and $\tau_0$ are in $H(\Gamma_0,\alpha_0)$.  On the other hand, if $pq$ is not an oriented edge in $\Gamma_0$ then $pq$ is normal to $\Gamma_0$ at $p$ and $qp$ is normal to $\Gamma_0$ at $q$.  Hence in this case both $\tau_0(p)$ and $\tau_0(q)$ are $S$-multiples of $\alpha(pq)$ and again \eqref{eq:whatever1} is satisfied.
\end{proof}

We say that a Thom class is \emph{non-vanishing} if its image in the ordinary cohomology ring is non-zero.

\begin{proposition}
\label{prop:StrInt}
Let $(\Gamma,\alpha,\theta,\lambda)$ be a $d$-valent 1-skeleton in $\R^n$.  Then the following are equivalent.
\begin{enumerate}
\item $(\Gamma,\alpha,\theta,\lambda)$ is straight.
\item Some (hence every) vertex of $\Gamma$ has a non-vanishing Thom class.
\item There exist non-zero constants $\left\{c_p\right\}_{p\in V_\Gamma}$ such that for every $f\in H(\Gamma,\alpha)$ we have
$$\sum_{p\in V_\Gamma}\frac{f(p)}{c_p\cdot\prod_{e\in E^p}\alpha(e)}\in S.$$
\end{enumerate}
\end{proposition}

\begin{proof}

(iii) $\Rightarrow$ (ii).
Define the function $\int_\Gamma\colon\Maps(V_\Gamma,S)\rightarrow Q$ by $\int_\Gamma f=\sum_{p\in V_\Gamma}\frac{f(p)}{c_p\prod_{e\in E^p}\alpha(e)}$.  Assuming that (iii) holds, the function restricts to a $S$-module map $\int_\Gamma\colon H(\Gamma,\alpha)\rightarrow S[-d]$.  Fix $p\in V_\Gamma$, and let $T_p$ denote the Thom class for $p$ defined by $T_p(p)=\prod_{e\in E^p}\alpha(e)$.  Then $\int_\Gamma T_p=\frac{1}{c_p}\neq 0$.  Since $\int_\Gamma$ is an $S$-module map, it passes to a map on ordinary cohomology $\int_\Gamma\colon \overline{H(\Gamma,\alpha)}\rightarrow\R[-d]$, and since $\int_\Gamma\bar{T}_p=\overline{\int_\Gamma T_p}=\frac{1}{c_p}\neq 0$ we conclude that the ordinary class $\bar{T}_p$ is non-zero.

\paragraph{(ii) $\Rightarrow$ (i)}
For each vertex $p\in V_\Gamma$, we may take Thom class to be the one for which $T_p(p)=\prod_{e\in E^p}\alpha(e)$, scaling if necessary.  For each edge $pq\in E_\Gamma$ define its Thom class by 
\begin{equation}
\label{eq:EdgeClass}
\sigma_{pq}(x)=\begin{cases}
\prod_{\substack{e\in E^p\\ e\neq pq\\}}\alpha(e) & \text{if $x=p$}\\
\left|K_{pq}\right|\cdot\prod_{\substack{e\in E^q\\ e\neq qp\\}}\alpha(e) & \text{if $x=q$}\\
0 & \text{otherwise},\\
\end{cases}
\end{equation}  
where $\left|K_{pq}\right|=\prod_{e\in E^p}\lambda_{pq}(e)$.  Then for each $pq\in E_\Gamma$ we have $\left|K_{pq}\right|\cdot T_q=T_p-\alpha(pq)\cdot \sigma_{pq}$.  Inductively, for any path $\gamma_i\colon p=p_0\sra p_1\sra\cdots\sra p_i$, we have  
$$\left|K_{\gamma_i}\right|\cdot T_{p_i}=T_{p_0}-\alpha(p_0p_1)\sigma_{p_0p_1}-\cdots\left|K_{\gamma_{i-1}}\right|\alpha(p_{i-1}p_i)\sigma_{p_{i-1}p_i}.$$
In particular for any loop $\gamma\colon p_0\sra p_1\sra\cdots\sra p_N\sra p_0$, we have that 
\begin{equation}
\label{eq:Str1}
\left|K_\gamma\right|T_{p_0}=T_{p_0}-\sum_{i=0}^N\left|K_{\gamma_i}\right|\alpha(p_{i}p_{i+1})\sigma_{p_ip_{i+1}}.
\end{equation}
Note that the sum on the RHS of \eqref{eq:Str1} is in the ideal $S^+\cdot H(\Gamma,\alpha)$.  Thus if $\bar{T}_{p_0}\neq 0$ then we must have that $\left|K_\gamma\right|=1$. 

\paragraph{(i) $\Rightarrow$ (iii)}
Fix a base point $p_0$ and define $p_0\coloneqq 1$.  Then for each $p_0\neq q\in V_\Gamma$ define a path $\gamma_q$ from $p_0$ to $q$, and set $c_q\coloneqq\left|K_{\gamma_q}\right|$.  In particular, straightness guarantees that for every $pq\in E_\Gamma$, we have $c_q=c_p\cdot \prod_{e\in E^p}\lambda_{pq}(e)$.  Fix $f\in H(\Gamma,\alpha)$.  The following argument has been lifted almost verbatim from the one given by Guillemin and Zara \cite[Theorem 2.2]{GZ0}.
  
  Let $\alpha_1,\ldots,\alpha_N$ denote a set of pairwise linearly independent vectors in $\R^n$ such that for each $e\in E_\Gamma$, $\alpha(e)$ is collinear with some $\alpha_j$, $1\leq j\leq N$.  Then we can find a suitable polynomial $g\in S$ such that 
\begin{equation}
\label{eq:integral2}
\sum_{p\in V_\Gamma}\frac{f(p)}{c_p\prod_{e\in E^p}\alpha(e)}=\frac{g}{\alpha_1\cdots\alpha_N}.
\end{equation}
We will show that for a fixed index $1\leq i\leq N$, $\alpha_i$ necessarily divides $g$.  Since $\alpha_1,\ldots,\alpha_N$ are relatively prime in $S$, this will show that RHS of \eqref{eq:integral2} does indeed lie in $S$.  Write $V_\Gamma=V_1\sqcup V_2$ where $V_1$ denotes the set of vertices $q$ such that for some $e\in E^q$, $\alpha(e)$ and $\alpha_i$ are collinear.  Note that the vertices in $V_1$ come in pairs.  Indeed if $p\in V_1$ then there is a unique $e\in E^p$ with $\alpha(e)=\lambda\cdot\alpha_i$, hence $q\coloneqq t(e)\in V_1$ as well.  Thus \eqref{eq:integral2} can be written as 
\begin{equation}
\label{eq:integral35}
\sum_{p\in V_1}\frac{f(p)}{c_p\prod_{e\in E^p}\alpha(e)}+\sum_{p\in V_2}\frac{f(p)}{c_p\prod_{e\in E^p}\alpha(e)}=\frac{g}{\alpha_1\cdots\alpha_N}
\end{equation}
and 
\begin{equation}
\label{eq:integral3}
\sum_{p\in V_1}\frac{f(p)}{c_p\prod_{e\in E^p}\alpha(e)}=\frac{1}{2}\sum_{\substack{pq\in E_\Gamma\\ \alpha(pq)=\lambda\alpha_i\\}}\left(\frac{f(p)}{c_p\prod_{e\in E^p}\alpha(e)}+\frac{f(q)}{c_q\prod_{e\in E^q}\alpha(e)}\right).
\end{equation}
Rewriting RHS of \eqref{eq:integral3} we get 
\begin{equation}
\label{eq:integral4}
\frac{1}{2}\sum_{\substack{pq\in E_\Gamma\\ \alpha(pq)=\lambda\alpha_i\\}}\left(\frac{f(p)c_q\prod_{\substack{e\in E^q\\ e\neq qp\\}}\alpha(e)-f(q)c_p\prod_{\substack{e\in E^p\\ e\neq pq\\}}\alpha(e)}{c_pc_q\cdot\alpha(pq)\cdot\prod_{\substack{e\in E^p\\ e\neq pq\\}}\alpha(e)\cdot\prod_{\substack{e\in E^q\\ e\neq qp\\}}\alpha(e)}\right).
\end{equation}

We observe that 
$$f(p)\equiv f(q) \ \ \ \text{mod} \ \alpha(pq)$$
and
$$c_p\prod_{\substack{e\in E^p\\ e\neq pq\\}}\alpha(e)\equiv c_q\prod_{\substack{e\in E^q\\ e\neq qp\\}}\alpha(e) \ \ \ \text{mod} \ \alpha(pq),$$
from which it follows that
$$f(p)c_q\prod_{\substack{e\in E^q\\ e\neq qp\\}}\alpha(e)\equiv f(q)c_p\prod_{\substack{e\in E^p\\ e\neq pq\\}}\alpha(e) \ \ \ \text{mod} \ \alpha(pq).$$
Hence the sum in \eqref{eq:integral4} can be rewritten as  
\begin{equation}
\label{eq:integral5}
\frac{1}{2}\sum_{\substack{pq\in E_\Gamma\\ \alpha(pq)=\lambda\alpha_i\\}}\frac{g_{pq}}{h_{pq}}
\end{equation}
for some $g_{pq},h_{pq}\in S$, where $h_{pq}$ is relatively prime to $\alpha_i$.

Plugging this observation into \eqref{eq:integral35} we get 
\begin{equation}
\label{eq:integral6}
\frac{1}{2}\sum_{\substack{pq\in E_\Gamma\\ \alpha(pq)=\lambda\alpha_i\\}}\frac{g_{pq}}{h_{pq}}+\sum_{p\in V_2}\frac{f(p)}{c_p\prod_{e\in E^p}\alpha(e)}=\frac{g}{\alpha_1\cdots\alpha_N}.
\end{equation}

Now since every denominator on the LHS of \eqref{eq:integral6} is relatively prime to $\alpha_i$, it follows that $\alpha_i$ must divide $g$.

\end{proof}

The function $\int_\Gamma\colon H(\Gamma,\alpha)\rightarrow S[-d]$ is called an \emph{integral operator} on $(\Gamma,\alpha,\theta,\lambda)$, and will play a fundamental role in this paper.  One interesting property of the integral operator is the following ``duality'' property, pointed out by Guillemin and Zara \cite{GZ2}:  

\begin{proposition}
\label{prop:StrDuality}
For a straight 1-skeleton $(\Gamma,\alpha,\theta,\lambda)$, an element $f\in\Maps(V_\Gamma,S)$ is in $H(\Gamma,\alpha)$ if and only if $\int_\Gamma f\cdot h\in S$ for all $h\in H(\Gamma,\alpha)$.
\end{proposition}
\begin{proof}
Fix $f\in \Maps(V_\Gamma,S)$, and suppose that $\int_\Gamma f\cdot h\in S$ for every $h\in H(\Gamma,\alpha)$.  Let $pq\in E_\Gamma$ be any edge, and let $h$ denote its Thom class with $h(p)\coloneqq\prod_{\substack{e\in E^p\\ e\neq pq\\}}\alpha(e)$.  Then we have 
\begin{align*}
\int_\Gamma f\cdot h= & \frac{f(p)}{c_p\alpha(pq)}+\frac{\prod_{e\in E^p}\lambda_{pq}(e)f(q)}{c_q\alpha(qp)}\\
= & \frac{f(p)-f(q)}{c_p\alpha(pq)},
\end{align*}
where the second inequality follows from the identity $\frac{c_q}{c_p}=\prod_{e\in E^p}\lambda_{pq}(e)$.  Hence if $\int_\Gamma f\cdot h\in S$, we must have that $f(p)-f(q)=c_{pq}\alpha(pq)$.  The other implication is simply (iii) in Proposition \ref{prop:StrInt}.
\end{proof}

\subsubsection{Generating Classes, Weak Generating Classes, and the Morse Package}
Fix a generic polarizing covector $\xi$, and a compatible Morse function $\phi$.  Define the \emph{flow-up at $p$} to be those vertices that are endpoints of some $\xi$-oriented path starting at $p$, denoted by $\mathcal{F}_p$.  Denote by $F_p$ those vertices whose $\phi$-value is at least the $\phi$-value of $p$.  Note that in general we have the containment $\mathcal{F}_p\subseteq F_p$.  

A homogeneous equivariant class $\tau_p\in H(\Gamma,\alpha)$ of degree $\ind_\xi(p)$ is called a \emph{generating class} for $p$ if
\begin{enumerate}
\item ${\dsp\supp(\tau_p)\subseteq\mathcal{F}_p}$, and 
\item ${\dsp \tau_p(p)=\prod_{e\in E^p_-}\alpha(e).}$
\end{enumerate}

A homogeneous equivariant class $\tau_p\in H(\Gamma,\alpha)$ of degree $\ind_{\xi}(p)$ is called a \emph{weak generating class} for $p$ if
\begin{enumerate}
\item ${\dsp \supp(\tau_p)\subseteq F_p}$, and
\item ${\dsp \tau_p(p)=\prod_{e\in E^p_-}\alpha(e).}$
\end{enumerate}

A \emph{generating family}, resp. \emph{weak generating family}, is a collection of generating classes, resp. weak generating classes, one for each vertex $\left\{\tau_p\right\}_{p\in V_\Gamma}$.  

The following result is essentially due to Guillemin and Zara \cite[Theorem 2.4.2]{GZ1}.
\begin{proposition}
\label{prop:MorsePackage}
The following are equivalent:
\begin{enumerate}
\item $H(\Gamma,\alpha)$ is a free $S$-module with $b_i(\Gamma,\alpha)$ generators in degree $i$.
\item $(\Gamma,\alpha,\theta,\lambda)$ admits a weak generating family.
\item $(\Gamma,\alpha,\theta,\lambda)$ admits a generating family.
\end{enumerate}
\end{proposition}
\begin{proof}
{(i) $\Rightarrow$ (ii).}  Label the vertices $V_\Gamma=\left\{p_1,\ldots,p_N\right\}$ so that $\phi(p_1)<\cdots<\phi(p_N)$.  Define the submodule filtration of $H(\Gamma,\alpha)$ by setting $H_0\coloneqq H(\Gamma,\alpha)$, $H_{N}\coloneqq \left\{0\right\}$, and for each $1\leq i\leq N-1$ define submodule  
$$H_{i}\coloneqq\left\{f\in H(\Gamma,\alpha)\left| \supp(f)\subseteq F_{p_{i+1}}\right.\right\}.$$
Then for each $0\leq i\leq N-1$ we have an exact sequence of graded vector spaces
\begin{equation}
\label{eq:MorseExact}
\xymatrix{ 0\ar[r] & H_{i+1}\ar[r] & H_{i}\ar[r]^-{\epsilon_{i}} & \prod_{e\in E^{p_{i+1}}_-}\alpha(e)\cdot S}
\end{equation}
where the first map is inclusion, and second map, $\epsilon_{i+1}$, is evaluation at the vertex $p_{i+1}$.  Restricting to the $m^{th}$ graded pieces in \eqref{eq:MorseExact} we get the inequalities:
\begin{equation}
\label{eq:MorseIneq}
\dim_\R\left(H_i^m(\Gamma,\alpha)\right)-\dim_\R\left(H^m_{{i+1}}\right)\leq\dim_\R\left(S^{m-\sigma_{p_{i+1}}}\right), \ \ 0\leq i\leq N-1.
\end{equation}
Summing \eqref{eq:MorseIneq} over $i$, we get the inequality
\begin{align}
\label{eq:MorseIneq2}
\dim_\R\left(H^m(\Gamma,\alpha)\right)\leq & \sum_{p\in V_\Gamma}\dim_\R\left(S^{m-\sigma_p}\right)\\
\nonumber= & \sum_{j=0}^d b_j(\Gamma,\alpha)\dim_\R\left(S^{m-j}\right).
\end{align}
If $H(\Gamma,\alpha)$ is a free $S$-module with $b_i(\Gamma,\alpha)$ generators in degree $i$, then the inequality in \eqref{eq:MorseIneq2} must be an equality.  But this implies that the inequalities in \eqref{eq:MorseIneq} are also equalities, which, in turn implies that the evaluation map $\epsilon_{p_{i+1}}$ must be surjective for $0\leq i\leq N-1$.    On the other hand, the surjectivity of $\epsilon_{p_{i+1}}$ exactly means that vertex $p_{i+1}$ has a weak generating class for $0\leq i\leq N-1$.

{(ii) $\Rightarrow$ (iii).}  Let $\left\{\tau_p\right\}_{p\in V_\Gamma}$ be a weak generating family for $(\Gamma,\alpha,\theta,\lambda)$, and label the vertices as before, with $\phi(p_1)<\cdots<\phi(p_N)$.  We will construct a generating family $\left\{\kappa_p\right\}_{p\in V_\Gamma}$ from the weak generating family $\left\{\tau_p\right\}_{p\in V_\Gamma}$.  Note that $\kappa_{p_N}\coloneqq\tau_{p_N}$ is already a generating class for $p_N$, since $\mathcal{F}_{p_N}=F_{p_N}=\left\{p_N\right\}$.  Inductively, we assume that we have generating classes $\kappa_{p_N},\kappa_{p_{N-1}},\ldots,\kappa_{p_{k+1}}$ for the vertices $p_N,\ldots,p_{k+1}$, respectively, and we will construct a generating class for $p_k$ starting from the weak generating class $\tau_{p_k}$.  Note that if $\supp\left(\tau_{p_k}\right)\subseteq\mathcal{F}_{p_k}\subseteq F_p$ then $\kappa_{p_k}\coloneqq\tau_{p_k}$ is already a generating class and we are done.  Otherwise there is a vertex $q_0$ in $\supp\left(\tau_{p_k}\right)\setminus \mathcal{F}_{p_k}$ whose $\phi$-value is smallest.  If $q_0t_1,\ldots,q_0t_r$ are the oriented edges at $q_0$ that flow into $q_0$, note that by our choice of $q_0$ we must have $\tau_{p_k}(t_i)=0$ for $1\leq i\leq r$.  Hence there exists a $c_0\in S$ such that 
$$\tau_{p_k}(q_0)=c_0\prod_{e\in E^{q_0}_-}\alpha(e).$$
Since $q_0\in F_{p_k}$ we must have $\phi(p_k)<\phi(q_0)$, hence, by our inductive hypothesis, $q_0$ has a generating class $\kappa_{q_0}$.  Define the new class 
$$\tau_{p_k,1}\coloneqq\tau_{p_k}-c_0\kappa_{q_0}.$$
Now if $\supp\left(\tau_{p_k,1}\right)\subseteq\mathcal{F}_{p_k}$ we are done.  Otherwise we proceed as before, and choose the $\phi$-smallest vertex $q_1\in\supp\left(\tau_{p_k,1}\right)\setminus\mathcal{F}_{p_k}$.  Note that $\phi(q_0)<\phi(q_1)$, since $\supp\left(\tau_{p_k,1}\right)=\left(\supp\left(\tau_{p_k}\right)\cup\supp\left(\kappa_{p_k}\right)\right)\setminus\{q_0\}$.  Hence this process must terminate after finitely many, say $\ell$, iterations yielding a homogeneous class
$$\tau_{p_k,\ell}\coloneqq\tau_{p_k}-c_0\kappa_{q_0}-c_1\kappa_{q_1}-\cdots-c_{\ell-1}\kappa_{q_{\ell-1}}$$
which must be a generating class for $p_k$.

{(iii) $\Rightarrow$ (i).}  Let $\left\{\tau_p\right\}_{p\in V_\Gamma}$ be a generating family.  Suppose we have a nontrivial $S$-linear dependence relation 
\begin{equation}
\label{eq:DepRel}
\sum_{p\in V_\Gamma}c_p\tau_p=0,
\end{equation}
and let $p_0$ be the smallest vertex for which $c_{p_0}\neq 0$.  Then we must have $\tau_{p_0}=-\frac{1}{c_{p_0}}\left(\sum_{\phi(q)>\phi(p_0)} c_q\tau_q\right)$.  But this is impossible since $p_0\notin\supp(\tau_q)$ for $\phi(q)>\phi(p_0)$.  Therefore the generating family must be $S$-linearly independent.  To see that the generating family spans, fix a homogeneous $f\in H(\Gamma,\alpha)$, and let $p$ be the smallest vertex in $\supp(f)$.  Then $f(p)$ must equal $s_p\cdot\prod_{e\in E^p_-}\alpha(e)$ for some $s_p\in S$, and $f-s_p\cdot\tau_p$ is therefore another class which is supported on vertices strictly larger than $p$.  Proceeding this way, we will eventually run out of vertices and end up with a class with empty support, i.e. the zero class.  In other words we will have $f-\sum_{p\in V_\Gamma}s_p\tau_p=0$, which proves that the generating family spans $H(\Gamma,\alpha)$.
This shows that $H(\Gamma,\alpha)$ is a free $S$-module with a basis $\left\{\tau_p\right\}_{p\in V_\Gamma}$.  Since the degree of $\tau_p$ is $\ind_\xi(p)$, we see that there are $b_i(\Gamma,\alpha)$ generators in degree $i$.
\end{proof}

\begin{definition}
\label{def:MorsePackage}
A 1-skeleton $(\Gamma,\alpha,\theta,\lambda)$ in $\R^n$ that satisfies any of the conditions in Proposition \ref{prop:MorsePackage} is said to have the \emph{Morse package}.
\end{definition}

Note that if $(\Gamma,\alpha,\theta,\lambda)$ has the Morse package, then $b_0(\Gamma,\alpha)=1$, i.e. the 1-skeleton is pointed.  Indeed any equivariant class of degree zero must be a constant (since $\Gamma$ is connected), hence any two classes in $H^0(\Gamma,\alpha)$ must be constant multiples of each other.  Thus, up to a constant, $H(\Gamma,\alpha)$ has exactly one generator in degree zero.  By symmetry of the combinatorial Betti numbers we deduce that $b_d(\Gamma,\alpha)=1$ as well.

\begin{proposition}
\label{prop:StMor}
If $(\Gamma,\alpha,\theta,\lambda)$ has the Morse package, then it is straight.
\end{proposition}
\begin{proof}
Since $b_d(\Gamma,\alpha)=1$ we know that there is a non-vanishing Thom class on a vertex of $\Gamma$; it is the generating class for the unique maximum vertex with respect to the Morse function $\phi$.  Hence by Proposition \ref{prop:Straightness}, $(\Gamma,\alpha,\theta,\lambda)$ must be straight.
\end{proof}

In fact we also have the following:
\begin{proposition}
\label{prop:StMor2}
If every $2$-slice of $(\Gamma,\alpha,\theta,\lambda)$ has the Morse package, then $(\Gamma,\alpha,\theta,\lambda)$ must be straight.
\end{proposition}
\begin{proof}
A $2$-slice that has the Morse package must be straight and pointed from our discussion above.  Hence a 1-skeleton whose $2$-slices have the Morse package must be non-cyclic with straight $2$-slices.  It follows from Proposition \ref{prop:Straightness} that the 1-skeleton must be straight itself.
\end{proof}

\subsection{Residues}
Following \cite{GZ0}:
Let $A$ be any integral domain, and let $A[x]$ be the ring of polynomials in one variable with coefficients in $A$.  For fixed $z\in A$, the fraction $\frac{1}{x-z}$ has a unique formal power series expansion about $x=\infty$: 
\begin{align}
\label{eq:FormalLS}
\nonumber\frac{1}{x-z}= & x^{-1}\left(\sum_{n\geq 0}z^nx^{-n}\right)\\
= & \sum_{n\geq 1}z^{n-1}x^{-n}.
\end{align}
Hence for any $f(x)\in A[x]$ and any $z_1,\ldots,z_m\in A$, the fraction $\frac{f(x)}{(x-z_1)\cdots(x-z_m)}$ has a unique formal Laurent series expansion about $x=\infty$:
\begin{align}
\label{eq:FormalLS1}
\nonumber\frac{f(x)}{(x-z_1)\cdots(x-z_m)}= & f(x)\cdot\left(\sum_{n\geq 1} z_1^{n-1}x^{-n}\right)\cdots\left(\sum_{n\geq 1}z_m^{n-1}x^{-n}\right)\\
= & \sum_{n\geq -N}a_nx^{-n}.
\end{align}
    Define the \emph{residue at $x=\infty$} of the quotient $\frac{f(x)}{(x-z_1)\cdots(x-z_m)}$ to be the coefficient of $x^{-1}$ in its Laurent expansion, i.e.
$$\Res_{x=\infty}\left(\frac{f(x)}{(x-z_1)\cdots(x-z_m)}\right)\coloneqq\Res_{x=\infty}\left(\sum_{n\geq -N} a_nx^{-n}\right)\coloneqq a_1.$$

The following two key facts are due to Guillemin and Zara \cite{GZ0}, and we refer the reader there for the proofs.

\begin{proposition}
\label{thm:Residue}
In the notation above, if $z_1,\ldots,z_m$ are distinct elements of $A$, then we have
\begin{equation}
\label{eq:Residue}
\Res_{x=\infty}\left(\frac{f(x)}{(x-z_1)\cdots(x-z_m)}\right)=\sum_{k=1}^m\frac{f(z_k)}{\prod_{j\neq k}(z_k-z_j)}.
\end{equation}
\end{proposition}
\begin{proof}
See \cite[Lemma 1]{GZ0}.
\end{proof}

\begin{proposition}
\label{prop:R2}
The quotient $h(x)=\frac{f(x)}{(x-z_1)\cdots(x-z_m)}$ is in $A[x]$ if and only if $\Res_\xi\left(x^kh(x)\right)=0$ for all $k\geq 0$.
\end{proposition}
\begin{proof}
See \cite[Lemma 2]{GZ0}.
\end{proof}

In our applications, $A$ will be a graded ring.  To distinguish between the grading on $A$, and the natural grading on $A[x]$, we define the \emph{$x$-degree} of $f(x)\in A[x]$ in the usual way, as the largest non-negative integer $N$ for which $f(x)=\sum_{j=0}^N a_jx^j$ with $a_j\in A$ and $a_N\neq 0$.  

\begin{lemma}
\label{lem:xdeg}
If the $x$-degree of $f(x)\in A[x]$ is less than $m-1$, then 
$$\Res_{x=\infty}\left(\frac{f(x)}{(x-z_1)\cdots(x-z_m)}\right)=0.$$
\end{lemma}
\begin{proof}
It suffices to show this for $f(x)=x^N$.  By \eqref{eq:FormalLS} we have
\begin{equation}
\label{eq:xdeg}
\frac{x^N}{(x-z_1)\cdots(x-z_m)}=x^{N-m}\left(\sum_{n=0}^\infty z_1^nx^{-n}\right)\cdots\left(\sum_{n=0}^\infty z_m^nx^{-n}\right)=x^{N-m}\cdot\sum_{n=0}^\infty a_nx^{-n}.
\end{equation}
Now if $N-m<-1$, it follows that $a_{-1}$ must equal zero in \eqref{eq:xdeg}.
\end{proof}

Proposition \ref{thm:Residue} and Lemma \ref{lem:xdeg} yield some useful identities in the field of fractions $Q(A)$.

\begin{corollary}
\label{cor:11}
For any distict elements $z_1,\ldots,z_m$ in $A$, we have 
\begin{equation}
\label{eq:11}
\sum_{i=1}^m\frac{z_i^{m-1}}{\prod_{j\neq i}(z_i-z_j)}=1.
\end{equation}
\end{corollary}
\begin{proof}
By \eqref{eq:FormalLS} the Laurent series for the rational function $\frac{x^{m-1}}{(x-z_1)\cdots(x-z_m)}$ is the product given by 
\begin{equation}
\label{eq:LS1}
x^{m-1}\cdot x^{-m}\left(\sum_{n=0}^\infty z_1^n x^{-n}\right)\cdots\left(\sum_{n=0}^\infty z_m^n x^{-n}\right).
\end{equation}
The coefficient of the $-1$-term in the formal Laurent series resulting from \eqref{eq:LS1} is clearly $1$, the RHS of \eqref{eq:11}.  On the other hand Theorem \ref{thm:Residue} implies that the residue of the rational function $\frac{x^{m-1}}{(x-z_1)\cdots(x-z_m)}$ is $\sum_{i=1}^m\frac{z_i^{m-1}}{\prod_{j\neq i}(z_i-z_j)}$, the RHS of \eqref{eq:11}, as claimed.
\end{proof}

\begin{corollary}
\label{cor:12}
For any distinct elements $z_1,\ldots,z_m$ in $A$, we have
\begin{equation} 
\label{eq:12}
\sum_{i=1}^m\frac{\prod_{j\neq i} z_j}{\prod_{j\neq i}(z_j-z_i)}=1.
\end{equation}
\end{corollary}
\begin{proof}
Consider the rational function $\frac{\prod_{i=1}^m z_i}{x\cdot (x-z_1)\cdots(x-z_m)}$.  Note that by Lemma \ref{lem:xdeg}, we have
\begin{equation}
\label{eq:Res2}
\Res_{x=\infty}\left(\frac{\prod_{i=1}^m z_i}{x\cdot (x-z_1)\cdots(x-z_m)}\right)=0.
\end{equation}
On the other hand, by Theorem \ref{thm:Residue} we have that 
\begin{equation}
\label{eq:Res3}
\Res_{x=\infty}\left(\frac{\prod_{i=1}^m z_i}{x\cdot (x-z_1)\cdots(x-z_m)}\right)=\frac{\prod_{j=1}^m z_j}{\prod_{j=1}^m(-z_j)}+\sum_{i=1}^m\frac{\prod_{j=1}^m z_j}{z_i\cdot \prod_{j\neq i}(z_i-z_j)}.
\end{equation}
Combining \eqref{eq:Res2} and \eqref{eq:Res3} yields the identity
\begin{equation}
\label{eq:Res4}
(-1)^m\left(1-\sum_{i=1}^m\frac{\prod_{j\neq i} z_j}{\prod_{j\neq i}(z_j-z_i)}\right)=0,
\end{equation}
from which the result follows.
\end{proof}

For a fixed polarizing covector $\xi\in\left(\R^n\right)^*$ fix a basis for $\R^n$ $x,y_1,\ldots,y_{n-1}$ such that $\langle\xi,x\rangle=1$ and $\langle\xi,y_i\rangle=0$ for $1\leq i\leq n-1$, i.e. $y_1,\ldots,y_{n-1}$ is a basis for $W_\xi$, the annihilator subspace of $\xi\colon\R^n\rightarrow\R$.
We regard the polynomial ring $S$ as polynomials in the variable $x$ with coefficients polynomials in $y_1,\ldots,y_{n-1}$, i.e. $S\cong S_\xi[x]$, where $S_\xi\coloneqq\Sym(W_\xi)$.  For any vectors $\alpha_1,\ldots,\alpha_m$ in $\R^n$, we have $\alpha_i=m_i\left(x-\beta_i\right)$, for $1\leq i\leq m$, where the $\beta_1,\ldots, \beta_m$ are elements of $S_\xi$.  
Then for any $f\coloneqq f(x,{\mathbf{y}})\in S\cong S_\xi[x]$ we define the \emph{residue (with respect to $\xi$)} of the quotient $\frac{f(x,{\mathbf{y}})}{\alpha_1\cdots\alpha_m}$ by
\begin{equation}
\label{eq:Res1}
\Res_{\xi}\left(\frac{f(x,{\mathbf{y}})}{\alpha_1\cdots\alpha_m}\right)\coloneqq\frac{1}{\prod_{i=1}^m m_i}\Res_{x=\infty}\left(\frac{f(x,{\mathbf{y}})}{(x-\beta_1)\cdots(x-\beta_m)}\right).
\end{equation}

In particular, $\Res_{\xi}\left(\frac{f(x,{\mathbf{y}})}{\alpha_1\cdots\alpha_m}\right)$ is always a polynomial in $S_\xi$.  For maps $f\colon V\rightarrow S_\xi[x]$ we will use the notation $f_p(x,{\mathbf{y}})$ for the polynomial value of $f$ at vertex $p$; we may also use the usual notation $f(p)$ if the polynomial variables are understood.

\subsection{The Kirwan Map and Cross Sectional Cohomology}
Fix a $d$-valent 1-skeleton $(\Gamma,\alpha,\theta,\lambda)$ in $\R^n$, fix a generic polarizing covector $\xi\in\left(\R^n\right)^*$, and fix a compatible Morse function $\phi\colon V_\Gamma\rightarrow\R$.  The vertices $V_\Gamma$ can be regarded as ``critical points'' of the Morse function $\phi$, and their images are the critical values.  We call the complement $\R\setminus\phi(V_\Gamma)$ the set of \emph{$\phi$-regular values}.  
 
Given a $\phi$-regular value $c\in\left(\phi_{\min},\phi_{\max}\right)$, define the \emph{$c$-vertex set} as the oriented edges of $\Gamma$ at $c$-level:
$$V_c\coloneqq\left\{pq\in E_\Gamma\left|\right. \phi(p)<c<\phi(q)\right\}.$$
For each $e\in E_\Gamma$ the linear map $\rho_e\colon \R^n\rightarrow W_\xi$ defined by 
$$\rho_e(x)=x-\frac{\langle\xi,x\rangle}{\langle\xi,\alpha(e)\rangle}\alpha(e)$$
extends to a map of symmetric algebras $\rho_e\colon S\rightarrow S_\xi$.
Note that for each $e\in E_\Gamma$ we have $\rho_e\equiv\rho_{\bar{e}}$.

Define the map $\mathcal{K}_c\colon\Maps(V_\Gamma,S)\rightarrow\Maps(V_c,S_\xi)$ by 
$$\mathcal{K}_c(f)(e)\coloneqq\rho_e\left(f(i(e))\right).$$
Restricting $\mathcal{K}_c$ to $H(\Gamma,\alpha)\subseteq\Maps(V_\Gamma,S)$ defines the \emph{Kirwan map}.  Note that for $f\in H(\Gamma,\alpha)$, we have 
$$\rho_e\left(f(i(e))\right)=\rho_e\left(f(t(e))\right)=\rho_{\bar{e}}(f(i(\bar{e}))).$$  

Following \cite{GZ0,GZ2} and as above, we fix a basis $x,y_1,\ldots,y_{n-1}$ for $\R^n$, such that $\langle\xi,x\rangle=1$ and $y_1,\ldots,y_{n-1}$ spans $W_\xi\subset\R^n$.  For each $e\in E_\Gamma$ let $m_e\coloneqq\langle\xi,\alpha(e)\rangle$, and write 
$\alpha(e)=m_e\left(x-\beta_e\right),$
where $\beta_e\coloneqq\beta_e(y_1,\ldots,y_{n-1})$ is a linear form in the variables $y_1,\ldots,y_{n-1}$.  Note then that for any $f=f(x,\y)\in S\cong S_\xi[x]$ we have
$\rho_e(f(x,\y))=f(\beta_e,\y).$
In particular, we have $\rho_e(\alpha(e'))=m_{e'}\left(\beta_e-\beta_{e'}\right)$.  Note that the genericity of $\xi$ guarantees that the differences $\left\{\beta_{e}-\beta_{e'}\left|e\neq e'\in E^p\right.\right\}$ are pairwise distinct.

If $(\Gamma,\alpha,\theta,\lambda)$ is straight, we can find some constants $\left\{c_p\right\}_{p\in V_\Gamma}$ satisfying  
$$c_q=c_p\prod_{e\in E^p}\lambda_{pq}(e)=c_p\left|K_{pq}\right| \ \ \ \forall \ pq\in E_\Gamma$$  
as we showed in the proof of Proposition \ref{prop:StrInt}.  If we choose and fix some constants $\left\{c_p\right\}_{p\in V_\Gamma}$, we can define the cross sectional integral operator $\int_{\Gamma_c}\colon\Maps(V_c,S_\xi)\rightarrow Q_\xi$, by
\begin{equation}
\label{eq:cIntegral}
\int_{\Gamma_c}f\coloneqq\sum_{e\in V_c}\frac{f(e)}{c_{i(e)}m_e\Kirc\left(\sigma_e\right)(e)}=\sum_{e\in V_c}\frac{f(e)}{c_{i(e)}m_e\prod_{\substack{e'\in E^{i(e)}\\ e'\neq e}}\rho_e(\alpha(e'))},
\end{equation}
where $\sigma_e\in H^{d-1}(\Gamma,\alpha)$ is the Thom class for the oriented edge $e\in V_c$, and the constants $\left\{c_{i(e)}\right\}_{e\in V_c}$ are as above. 

The following beautiful result is due to Guillemin and Zara \cite[Theorem 2.5]{GZ0}.  The proof given in \cite{GZ0} applies almost verbatim, but we reproduce it here for the sake of completeness.  
\begin{lemma}
\label{prop:KirwanIntegral}
If $(\Gamma,\alpha,\theta,\lambda)$ is straight, with $\xi$, $\phi$, and $c$ fixed as above, then for each $f\in H(\Gamma,\alpha)$ we have 
\begin{equation}
\label{eq:ResidueIntegral1}
\int_{\Gamma_c}\Kirc(f)=\sum_{\phi(q)<c}\frac{1}{c_q}\Res_\xi\left(\frac{f_q(x,{\mathbf{y}})}{\prod_{e\in E^p}\alpha(e)}\right).
\end{equation}
In particular ${\dsp \int_{\Gamma_c}\Kirc(f)\in S_\xi.}$
\end{lemma}
\begin{proof}
Choose and fix $\phi$-regular values $c_0< c_1<\ldots c_N=c$ such that for any $0\leq i\leq N$ there is a unique vertex $p_i$ satisfying $c_{i-1}<\phi(p_i)<c_i$.  Then we compute 
\begin{equation}
\label{eq:ResidueIntegral2}
\int_{\Gamma_{c_i}}\Kir_{c_i}(f)-\int_{\Gamma_{c_{i-1}}}\Kir_{c_{i-1}}(f)=\sum_{e\in E^{p_i}_+}\frac{\rho_{e}(f(p_i))}{c_{p_i}m_e\rho_e(\sigma_e(p_i))}-
\sum_{e\in E^{p_i}_-}\frac{\rho_{\bar{e}}(f(i(\bar{e})))}{c_{i(\bar{e})}m_{\bar{e}}\rho_{\bar{e}}(\sigma_{\bar{e}}(i(\bar{e})))}.
\end{equation}
Note that $\rho_{\bar{e}}\equiv\rho_e$, and that $\rho_e(f(i(e)))=\rho_{\bar{e}}(f(i(\bar{e})))$ for any $f\in H(\Gamma,\alpha)$.  Also note that $m_e=-m_{\bar{e}}$, and that $\sigma_{\bar{e}}(i(\bar{e}))=\left|K_{\bar{e}}\right|\cdot\sigma_e(i(e))$.  Then RHS of \eqref{eq:ResidueIntegral2} becomes
\begin{equation}
\label{eq:ResidueIntegral3}
\sum_{e\in E^{p_i}_+}\frac{\rho_{e}(f(p_i))}{c_{p_i}m_e\rho_e(\sigma_e(p_i))}+ 
\sum_{e\in E^{p_i}_-}\frac{\rho_{e}(f(p_i))}{c_{i(\bar{e})}m_{e}\left|K_{\bar{e}}\right|\rho_{e}(\sigma_{e}(p_i))}.
\end{equation}
Finally noting that $\left|K_{\bar{e}}\right|\cdot c_{i(\bar{e})}=c_{i(e)}$ we can write \eqref{eq:ResidueIntegral3} as
\begin{equation}
\label{eq:ResInt4}
\sum_{e\in E^{p_i}}\frac{\rho_e(f(p_i))}{c_{p_i}m_e\rho_e(\sigma_e(p_i))}=\frac{1}{c_{p_i}}\cdot\Res_\xi\left(\frac{f(p_i)}{\prod_{e\in E^{p_i}}\alpha(e)}\right)
\end{equation}
Since the integral ${\dsp \int_{\Gamma_{c_0}}\Kir_{c_0}(f)=0}$ (it is a sum over the empty set) we deduce the formula in \eqref{eq:ResidueIntegral1}.  
\end{proof}

Following Guillemin and Zara \cite[Definition 5.1]{GZ2}, we use the result of Proposition \ref{prop:KirwanIntegral} to define the equivariant cross sectional cohomology for the $c$-cross section of a 1-skeleton.

\begin{definition}
\label{def:Hc}
Define $H(\Gamma_c)$ to be the subset $g\in\Maps(V_c,S_\xi)$ such that
$$\int_{\Gamma_c}g\cdot\Kirc(h)\in S_\xi, \ \ \ \ \ \ \text{for every $h\in H(\Gamma,\alpha)$}.$$
\end{definition}
Note that by $S_\xi$-linearity of $\Gamma_\xi$, the set $H(\Gamma_c)$ is an $S_\xi$ submodule of $\Maps(V_c,S_\xi)$.  Moreover Lemma \ref{prop:KirwanIntegral} implies that $\Kirc\left(H(\Gamma,\alpha)\right)$ is a subset of $H(\Gamma_c)$.  Thus the Kirwan map is really a map \emph{into} $H(\Gamma_c)\subset\Maps(V_c,S_\xi)$, i.e.
$\Kirc\colon H(\Gamma,\alpha)\rightarrow H(\Gamma_c).$

\subsection{Complete Graphs and Initial Cross Sectional Cohomology}
Following \cite[Section 4]{GZ2}:  Let $\Delta=\left\{v_0,\ldots,v_d\right\}$ denote any finite set, and let $\tau\colon \Delta\rightarrow W_\xi$ denote any injective function with $\tau(v_i)\coloneqq\beta_i$.  Then on the set $\Maps(\Delta,S_\xi)$ there is an integral operator:
\begin{equation}
\label{eq:CompleteInt}
\int_\Delta g\coloneqq \sum_{i=0}^d\frac{g(v_i)}{\prod_{j\neq i}\left(\beta_i-\beta_j\right)}
\end{equation}
\begin{definition}
\label{def:HDc}
The subset $H(\Delta,\tau)\subseteq\Maps(\Delta,S_\xi)$ is the set of maps $g\colon\Delta\rightarrow S_\xi$ satisfying 
$$\int_{\Delta}g\cdot P(\tau) \ \in S_\xi \ \ \ \ \text{for each $P\in S_\xi[X]$}.$$
\end{definition}

\begin{lemma}
\label{lem:Lem4.1}
A map $g\colon\Delta\rightarrow S_\xi$ is in $H(\Delta,\tau)$ if and only if there exist $g_0,\ldots,g_{d}\in S_\xi$ such that 
\begin{equation}
\label{eq:CG}
g=\sum_{k=0}^dg_k\tau^k.
\end{equation}
\end{lemma}
\begin{proof}
See \cite[Theorem 4.1]{GZ2}.
\end{proof}

\begin{corollary}
\label{cor:Lem4.1}
The $S_\xi$ module $H(\Delta,\tau)$ is free with basis $\left\{\tau^i\left|\right. 0\leq i\leq d\right\}$.
\end{corollary}
\begin{proof}
This follows directly from Lemma \ref{lem:Lem4.1}.
\end{proof}

Note that by definition of $H(\Delta,\tau)$ the element $\tau^{d+1}\in\Maps(\Delta,S)$ must belong to $H(\Delta,\tau)$.  Thus by Lemma \ref{lem:Lem4.1} we deduce that 
\begin{equation}
\label{eq:taud}
\tau^{d+1}=\sum_{k=0}^d(-1)^{d-k}\sigma_{d+1-k}\cdot\tau^k
\end{equation}
for some polynomials $\sigma_{d+1-k}\in S^{d+1-k}$.  It turns out that the coefficient $\sigma_{d-k+1}=\sigma_{d-k+1}(\beta_0,\beta_1,\ldots,\beta_d)$ in \eqref{eq:taud} is the $d-k+1^{st}$ elementary symmetric polynomial.  We digress briefly from graphs to symmetric polynomials.

For a fixed positive integer $m$, and a fixed integer $1\leq k\leq m$ we define the $k^{th}$ \emph{elementary symmetric polynomial} by 
\begin{equation}
\label{eq:ESF}
s_{m,k}\coloneqq \sum_{1\leq i_1<\cdots<i_k\leq m}x_{i_1}\cdots x_{i_k},
\end{equation}
where the sum on the RHS is over all $k$-subsets $\left\{i_1<\cdots<i_k\right\}\subseteq\left\{1,\ldots,m\right\}$.  For each $m$, we set $s_{m,0}\coloneqq 1$.
It will be convenient to reference the following well known identity, which we state as a Lemma.
\begin{lemma}
\label{prop:Xn}
For each $m$, and for each fixed $1\leq i\leq m$ we have
\begin{equation}
\label{eq:Xn}
x_i^{m}=\sum_{j=0}^{m-1}(-1)^{m-j-1}s_{m,m-j}\cdot x_i^j.
\end{equation}
\end{lemma}
\begin{proof}
Simply plug in $t=x_i$ to the polynomial identity
$$\prod_{i=1}^m(t-x_i)= \sum_{j=0}^m(-1)^{m-j}\cdot s_{m,m-j}\cdot t^j.$$
\end{proof}

Now by Proposition \ref{prop:Xn} we see that the coefficient $\sigma_{d-k+1}=s_{d+1,d+1-k}$ is the $(d-k+1)^{st}$ elementary symmetric polynomial in the $(d+1)$ variables $\tau(v_0),\tau(v_1),\ldots,\tau(v_d)$, as claimed above.  

In particular, the finitely generated free $S$-submodule $H(\Delta,\tau)\subseteq\Maps(\Delta,S)$ is closed under multiplication, hence is also an $S$ subalgebra of $\Maps(\Delta,S)$.  

Now suppose that $c$ is a $\phi$-regular value for which there is a unique $p\in V_\Gamma$ such that $\phi(p)<c$.  The cross sectional cohomology is called \emph{initial} in this case.  See \cite[Example 6.1]{GZ2}.

\begin{proposition}
\label{prop:Initial}
The initial cross sectional cohomology $H(\Gamma_c)$ is isomorphic to $H(\Delta,\tau)$, where $\Delta\coloneqq E^p$ and $\tau\coloneqq\Kirc(x)\colon \Delta\rightarrow S_\xi$.
\end{proposition}
\begin{proof}
Note that $V_c=E^p=\Delta$.  Then for any $f\in\Maps(V_c,S_\xi)=\Maps(\Delta,S_\xi)$, and for any $h\in H(\Gamma,\alpha)$ we have 
\begin{align*}
\int_{\Gamma_c} f\cdot\Kirc(h)= & \sum_{e\in E^p}\frac{f(e)\cdot\Kirc(h)(e)}{c_pm_e\prod_{\substack{e'\in E^p\\ e'\neq e\\}}(\beta_{e'}-\beta_{e})}\\
= & \frac{1}{c_p\prod_{e\in E^p}m_e}\sum_{e\in E^p}\frac{f(e)\cdot \rho_e\left(h_p(x,{\mathbf{y}})\right)}{\prod_{\substack{e'\in E^p\\ e'\neq e\\}}(\beta_{e'}-\beta_{e})}\\
= & \frac{1}{c_p\prod_{e\in E^p}m_e}\sum_{e\in E^p}\frac{f(e)\cdot h_p(\beta_e,{\mathbf{y}})}{\prod_{\substack{e'\in E^p\\ e'\neq e\\}}(\beta_{e'}-\beta_{e})}\\
= & \frac{1}{c_p\prod_{e\in E^p}m_e}\sum_{e\in E^p}\frac{f(e)\cdot h_p(\tau(e),{\mathbf{y}})}{\prod_{\substack{e'\in E^p\\ e'\neq e\\}}(\beta_{e'}-\beta_{e})}\\
= & \frac{1}{c_p\prod_{e\in E^p}m_e}\int_{\Delta}f\cdot H_p(\tau),
\end{align*}
where $H_p(x)$ is the polynomial $h_p(x,{\mathbf{y}})$ viewed in the polynomial ring $S_\xi[x]\cong S$.  Furthermore note that the polynomial $H(x)\coloneqq H_p(x)\in S_\xi[x]$ may be chosen freely by choosing the equivariant class $h$ appropriately, e.g. we can always choose $h$ to be the constant class $q\mapsto H(x)\in S\cong S_\xi[x]$.  This implies that 
\begin{align*}
H(\Gamma_c)\coloneqq & \left\{f\in\Maps(V_c,S_\xi)\left| \int_{\Gamma_c}f\cdot\Kirc(h)\in S_\xi, \ \forall h\in H(\Gamma,\alpha)\right. \right\}\\
= & \left\{f\in\Maps(\Delta,S_\xi)\left| \int_\Delta f\cdot H(\tau)\in S_\xi \ \forall H(x)\in S_\xi[x]\right.\right\}\\
\eqqcolon & H(\Delta,\tau)
\end{align*}
as claimed.
\end{proof}

\section{Morse Theory}
\label{sec:Proof}
In this section, we fix a $d$-valent 1-skeleton $(\Gamma,\alpha,\theta,\lambda)$ in $\R^n$ which satisfies the acyclicity axiom.  We shall choose and fix a generic polarizing covector $\xi\in\left(\R^n\right)^*$ and a compatible Morse function $\phi\colon V_\Gamma\rightarrow\R$.  As before, we also fix a basis $x,y_1,\ldots,y_{n-1}$ for $\R^n$ such that $\langle x,\xi\rangle=1$ and $y_1,\ldots,y_{n-1}$ span the annihilator subspace $W_\xi\subset\R^n$.  As before, for every $e\in E_\Gamma$ we denote by $\rho_e\colon S\cong S_\xi[x]\rightarrow S_\xi$ the algebra map defined by $\rho_e\left(f(x,\y)\right)=f(\beta_e,\y)$.  

In subsections \ref{subsec:one}, \ref{subsec:two}, and \ref{subsec:three} we will also be assuming that every $2$-slice has the Morse package as in Theorem \ref{thm:GZMain}.  Note that under these assumptions $(\Gamma,\alpha,\theta,\lambda)$ must also be straight, by Proposition \ref{prop:StMor2}.  We therefore also choose and fix ``integration constants'' $\left\{c_p\right\}_{p\in V_\Gamma}$.

\subsection{Restriction Maps}
\label{subsec:one}
Fix two $\phi$-consecutive vertices and let $c$ be any regular value between them, say $\phi(q)<c<\phi(p)$.  Then the oriented edge sets $E^q_+$ and $\bar{E}^p_-\coloneqq\left\{\bar{e}\left|e\in E^p_-\right.\right\}$ are both subsets of $V_c$.  For each $e\in E_\Gamma$ let $m_e\coloneqq\langle\xi,\alpha(e)\rangle$, and in our fixed basis $x,y_1,\ldots,y_{n-1}$ write $\alpha(e)=m_e\left(x-\beta_e\right)$ where $\beta_e=\beta_e(\y)$ is a linear form in the variables $\y=y_1,\ldots,y_{n-1}$.  Note that $\beta_{\bar{e}}=\beta_e$.  Define the subsets $\Delta_{c}^-\coloneqq E^q_+$ and 
$\Delta_c^+\coloneqq\bar{E}^p_-$.
Define the maps $\tau_{c\pm}\colon\Delta_{c}^{\pm}\rightarrow S_\xi$ by $\tau_{c\pm}(e)\coloneqq \beta_e$.  Define the subrings $H(\Delta_c^{\pm})\coloneqq H(\Delta_c^{\pm},\tau_{c\pm})\subseteq\Maps(\Delta_c^\pm,S_\xi)$, as in Definition \ref{def:HDc}. 

\begin{proposition}
\label{prop:Thm6.1}
The set inclusions $\Delta_c^\pm\subseteq V_c$ define surjective restriction maps
\begin{equation}
\label{eq:restrictionmap}
r_c^\pm\colon H(\Gamma_c)\rightarrow H(\Delta_c^\pm).
\end{equation}
\end{proposition}
\begin{proof}
Proposition \ref{prop:Thm6.1} for GKM 1-skeleta is due to Guillemin and Zara \cite[Theorem 6.1]{GZ2}, and their ingenious proof applies almost verbatim.  In fact, the following is essentially just a rewording of their proof, tailored to fit our notation.

Note that if $r_c^\pm\left(H(\Gamma_c)\right)\subseteq H(\Delta_c^{\pm})$ then the restriction map in \eqref{eq:restrictionmap} must be surjective.  Indeed by Proposition \ref{lem:Lem4.1}, $H(\Delta_c^\pm)$ is generated as an $S_\xi$ algebra by the single element $\tau_{c\pm}\in H(\Delta_{c}^\pm)$.  Moreover, $r_c^\pm$ is an $S_\xi$-algebra map, and $r_c^\pm\left(\Kirc(x)\right)=\tau_{c\pm}$, where $x\in H^1(\Gamma,\alpha)$ is the constant function $V_\Gamma\ni v\mapsto x\in S^1$.  Thus it suffices to show that $r_c^\pm (f)\in H(\Delta_c^\pm)$ for every $f\in H(\Gamma_c)$.  We show this only for $r\coloneqq r_c^+$.  The proof for $r_c^-$ is completely analogous. 

We need to show that for each $f\in H(\Gamma_c)$, the rational function 
\begin{equation}
\label{eq:restrict1}
\int_{\Dc} r(f)\cdot P(\taupp)\coloneqq\sum_{e\in\Dc}\frac{f(e)\cdot P(\beta_e)}{\prod_{\substack{e'\in\Dc\\ e'\neq e\\}}\left(\beta_e-\beta_{e'}\right)}
\end{equation}
is actually a polynomial in $S_\xi$ for each $P(T)\in S_\xi[T]$.  

For each $\gamma\in \R^n$, let $\left(S_\xi\right)_\gamma$ denote the localized ring at the prime ideal generated by $\gamma$.  Let $\mathcal{M}\coloneqq\left\{\beta_i-\beta_j\left| 1\leq i<j\leq r\right.\right\}$.  In order to prove that the sum in \eqref{eq:restrict1} is in $S_\xi$, it suffices to prove that it lies in the local ring $\left(S_\xi\right)_\gamma$ for every direction $\gamma\in \mathcal{M}$.

Fix $\gamma\in \mathcal{M}$, and define an equivalence relation on $\Dc$ by declaring $e\equiv e'$ if and only if $\beta_e-\beta_{e'}$ is collinear with $\gamma$.  Let $\left(\Dc\right)_{\gamma}^1,\ldots,\left(\Dc\right)_{\gamma}^\ell$ be the distinct equivalence classes, so that in particular $\bar{E}_-^p=\bigsqcup_{k=1}^\ell\left(\Dc\right)_\gamma^k$.  Then we can decompose the sum on the RHS of \eqref{eq:restrict1} by
\begin{equation}
\label{eq:restrict4}
\sum_{e\in\Dc}\frac{f(e)\cdot P(\beta_e)}{\prod_{\substack{e'\in\Dc\\ e'\neq e\\}}\left(\beta_e-\beta_{e'}\right)}=\sum_{k=1}^\ell \left(\sum_{e\in\left(\Dc\right)_\gamma^k}\frac{f(e)\cdot P(\beta_e)}{\prod_{\substack{e'\in \Dc\\ e'\neq e\\}}\left(\beta_e-\beta_{e'}\right)}\right).
\end{equation}
We will show that each sum
\begin{equation}
\label{eq:restrict5}
I_{\gamma}^k\coloneqq\sum_{e\in\left(\Dc\right)_\gamma^k}\frac{f(e)\cdot P(\beta_e)}{\prod_{\substack{e'\in \Dc\\ e'\neq e\\}}\left(\beta_e-\beta_{e'}\right)}
\end{equation}
is in the local ring $\left(S_\xi\right)_\gamma$.

Fix $1\leq k\leq \ell$.  First, suppose that $\left(\Dc\right)_\gamma^k$ consists only of a single edge, say $e$.  Then for every other $e'\in \Dc$ different from $e$, the difference $\beta_e-\beta_{e'}$ must not be a multiple of $\gamma$.  Hence in this case the sum in \eqref{eq:restrict5} is in $\left(S_\xi\right)_{\gamma}$.  

Now suppose that $\left(\Dc\right)_\gamma^k$ contains at least two edges, say $e_1$ and $e_2$.  Let $H$ be the $2$-dimensional subspace spanned by $\alpha(e_{1})$ and $\alpha(e_{2})$, and let $(\Gamma_H^0,\alpha_H^0,\theta^0_H,\lambda^0_H)$ be the corresponding $2$-slice containing the vertex $p$.  Let $E^p_H\subseteq E^p$ denote those oriented edges at $p$ belonging to $\Gamma_H^0$, and let $N_H^p$ denote those oriented edges at $p$ that are normal to $\Gamma_H^0$.  Note that $e\in \left(\Dc\right)_\gamma^k$ if and only if $\bar{e}\in E^p_-\cap E^p_H$.
Indeed for any edge $e\in E^p_-$ we have  
$$\rho_{e_1}(\alpha(e))=\alpha(e)-\frac{\langle\xi,\alpha(e)\rangle}{\langle\xi,\alpha(e_{1})\rangle}\alpha(e_{1})=m_{e}\left(\beta_{e_1}-\beta_{e}\right) \in H\cap W_\xi,$$
and since $H\cap W_\xi=\R\cdot\gamma$, we see that $\alpha(e)\in H$ if and only if $\beta_{e_1}-\beta_e\in\R\cdot\gamma$.

By Propositions \ref{prop:ksliceNormStr} and \ref{prop:ThomNormal}, the $2$-slice $(\Gamma_H^0,\alpha_H^0,\theta^0_H,\lambda_H^0)$ supports a Thom class.  Scaling if necessary, we may choose that Thom class $T_H\in H(\Gamma,\alpha)$ such that 
$$T_H(p)=\prod_{e\in N^p_H}\alpha(e).$$
By our assumptions, $(\Gamma_H^0,\alpha_H^0,\theta^0_H,\lambda_H^0)$ has the Morse package.  Note that $-\xi|_H\in\left(H\right)^*$ provides a generic polarizing covector, and $-\phi\left|_{V_H^0}\right.$ provides a compatible Morse function.  Hence we can find a homogeneous class $\tau_{H,p}\in H(\Gamma_H^0,\alpha_H^0)$ for which $\supp\left(\tau_{H,p}\right)\subseteq\left\{x\in V_H^0\left| \phi(x)<\phi(p)\right.\right\}$ and for which 
$$\tau_{H,p}(p)=\prod_{e\in E^p_+\cap E^p_H}\alpha(e).$$

Note that by Proposition \ref{prop:whatever}, the product $\kappa_{H,p}\coloneqq T_H\cdot\tau_{H,p}\colon V_\Gamma\rightarrow S$ is also an equivariant class in $H(\Gamma,\alpha)$.  Hence for any polynomial $P\coloneqq P(T)\in S_\xi[T]\cong S$, the sum of rational functions
\begin{equation}
\label{eq:restrict6}
\sum_{e\in V_c}\frac{f(e)\cdot P(\beta_e)\cdot \rho_e\left(\kappa_{H,q}(i(e))\right)}{c_{i(e)}m_{e}\rho_e\left(\prod_{\substack{e'\in E^{i(e)}\\ e'\neq e\\}}\alpha(e')\right)}=\int_{\Gamma_c}f\cdot\Kirc\left(P\cdot \kappa_{H,p}\right),
\end{equation}
must be a polynomial in $S$, since $f\in H(\Gamma_c)$.
Note that for $e\in V_c\setminus\Dc$ we must have $\phi(t(e))>\phi(p)$ hence $\rho_e\left(\kappa_{H,p}(i(e))\right)=\rho_e\left(\kappa_{H,p}(t(e))\right)=\rho_e(0)=0$.  On the other hand, for $e\in \Dc$, we have 
$$\rho_e\left(\kappa_{H,p}(i(e))\right)=\rho_e\left(\kappa_{H,p}(p)\right)=
\begin{cases} 
\rho_e\left(\prod_{e'\in E^p_+\cap E^p_H}\alpha(e')\cdot\prod_{e''\in N^p_H}\alpha(e'')\right) & \text{if $e\in\left(\Dc\right)_\gamma^k$}\\ 
0 & \text{otherwise}\\ 
\end{cases}$$
We also have for each $e\in \Dc$,  
\begin{equation}
\label{eq:id1}
c_{i(e)}\rho_e\left(\prod_{\substack{e'\in E^{i(e)}\\ e'\neq e\\}}\alpha(e')\right)=c_p\rho_e\left(\prod_{\substack{e'\in E^{p}\\ e'\neq \bar{e}\\}}\alpha(e')\right).
\end{equation}
Therefore the integral in \eqref{eq:restrict6} simplifies to 
\begin{align}
\label{eq:restrict7}
\nonumber\int_{\Gamma_c}f\cdot\Kirc\left(P\cdot \kappa_{H,p}\right) & = \sum_{e\in V_c}\frac{f(e)\cdot P(\beta_e)\cdot \rho_e\left(\kappa_{H,p}(i(e))\right)}{c_{i(e)}m_{e}\rho_e\left(\prod_{\substack{e'\in E^{i(e)}\\ e'\neq e\\}}\alpha(e')\right)}\\
\nonumber & = \sum_{e\in \left(\Dc\right)_\gamma^k}\frac{f(e)\cdot P(\beta_e)}{c_pm_{e}\rho_e\left(\prod_{\substack{e'\in E^p_-\cap E^p_H\\ e'\neq e\\}}\alpha(e')\right)}\\
 & = -\frac{1}{c_qM_{H,q}}\cdot\sum_{e\in\left(\Dc\right)_\gamma^k}\frac{f(e)\cdot P(\beta_{e})}{\prod_{\substack{e'\in\left(\Dc\right)_\gamma^k\\ e'\neq e\\}}\left(\beta_{e}-\beta_{e'}\right)}
\end{align}
where $M_{H,p}\coloneqq \prod_{e\in E^p_-\cap E^p_H} m_{e}$.  Define the polynomial $R(T)=\prod_{e'\in\Dc\setminus\left(\Dc\right)_\gamma^k}(T-\beta_{e'})\in S_\xi[T]$.  We can rewrite the sum in \eqref{eq:restrict5} as
\begin{align}
\label{eq:restrict8}
\nonumber I_{\gamma}^k & \coloneqq\sum_{e\in\left(\Dc\right)_\gamma^k}\frac{f(e)\cdot P(\beta_e)}{\prod_{\substack{e'\in \Dc\\ e'\neq e\\}}\left(\beta_e-\beta_{e'}\right)}\\
\nonumber & = \sum_{e\in\left(\Dc\right)_\gamma^k}\frac{f(e)\cdot P(\beta_e)}{\prod_{\substack{e'\in \left(\Dc\right)_\gamma^k\\ e'\neq e\\}}\left(\beta_e-\beta_{e'}\right)}\cdot\frac{1}{\prod_{e'\in \Dc\setminus\left(\Dc\right)_\gamma^k}\left(\beta_e-\beta_{e'}\right)}\\
\nonumber & = \sum_{e\in\left(\Dc\right)_\gamma^k}\frac{f(e)\cdot P(\beta_e)}{\prod_{\substack{e'\in \left(\Dc\right)_\gamma^k\\ e'\neq e\\}}\left(\beta_e-\beta_{e'}\right)}\cdot\frac{1}{R(\beta_e)}\\
\nonumber & = \frac{1}{\prod_{e\in\left(\Dc\right)_\gamma^k}R\left(\beta_e\right)}\cdot \sum_{e\in\left(\Dc\right)_\gamma^k}\frac{f(e)\cdot P(\beta_e)\cdot\prod_{\substack{e'\in\left(\Dc\right)_\gamma^k\\e'\neq e\\}}R\left(\beta_{e'}\right)}{\prod_{\substack{e'\in \left(\Dc\right)_\gamma^k\\ e'\neq e\\}}\left(\beta_e-\beta_{e'}\right)}\\
\nonumber & = \frac{1}{\prod_{e\in\left(\Dc\right)_\gamma^k}R\left(\beta_e\right)}\cdot \sum_{e\in\left(\Dc\right)_\gamma^k}\frac{f(e)\cdot P(\beta_e)\cdot Q\left(\beta_e\right)}{\prod_{\substack{e'\in \left(\Dc\right)_\gamma^k\\ e'\neq e\\}}\left(\beta_e-\beta_{e'}\right)}\\
& = -\frac{c_qM_{H,q}}{\prod_{e\in\left(\Dc\right)_\gamma^k}R\left(\beta_e\right)}\cdot\int_{\Gamma_c}f\cdot\Kirc\left(P\cdot Q\right),
\end{align}
where $Q\coloneqq Q(T)\in S_\xi[T]$ satisfies 
\begin{equation}
\label{eq:restrict9}
Q(\beta_e)=\prod_{\substack{e'\in\left(\Dc\right)_\gamma^k\\e'\neq e\\}}R\left(\beta_{e'}\right), \ \ \ \forall \ e\in\left(\Dc\right)_\gamma^k.
\end{equation}
Before we explain why $Q(T)$ exists, let us note that if it does then the sum $I_{\gamma}^k$ must be in $\left(S_\xi\right)_\gamma$.  Indeed if there is such a $Q\in S$ then the integral $\int_{\Gamma_c}f\cdot\Kirc\left(P\cdot Q\right)$ must be a polynomial in $S_\xi$, and, since $R(\beta_e)\notin\gamma\cdot S$ for each $e\in\left(\Dc\right)_\gamma^k$, the fraction $-\frac{c_pM_{H,p}}{\prod_{e\in\left(\Dc\right)_\gamma^k}R\left(\beta_e\right)}$ must lie in the local ring $\left(S_\xi\right)_\gamma$.  Hence we can conclude that the product in \eqref{eq:restrict8} must also be in $\left(S_\xi\right)_\gamma$.

To finish the proof, we need to show that there is indeed a polynomial $Q(T)\in S_\xi[T]\cong S$ satisfying \eqref{eq:restrict9}.
To see this, let us label the edges in $\left(\Dc\right)_\gamma^k=\left\{e_1,\ldots,e_m\right\}$.  Define the polynomial 
$$\tilde{R}_i(X_1,\ldots,\hat{X_i},\ldots,X_{m})\coloneqq\prod_{\substack{j=1\\ j\neq i\\}}^{m} R(X_j),$$
where the ``hat'' symbol means omission. 
The symmetric group $\mathfrak{S}_{m-1}$ acts on the polynomial ring $S_\xi[X_1,\ldots,\hat{X_i},\ldots,X_{m}]\cong\R[y_1,\ldots,y_{n-1},X_1,\ldots,\hat{X_i},\ldots,X_{m}]$ trivially on the first $(n-1)$-variables and in the usual way on the last $(m-1)$-variables.  The invariant subalgebra $\left(S_{\xi}[X_1,\ldots,\hat{X_i},\ldots,X_{m}]\right)^{\mathfrak{S}_{m-1}}$ is evidently isomorphic to the polynomial subalgebra $S_\xi[P_{1,1},\ldots,\hat{P}_{i,i},\ldots,P_{i,m}]$, where $P_{i,j}$ is the $j^{th}$ power sum symmetric polynomial, i.e.
$$P_{i,j}=X_1^j+\cdots+\hat{X}_i^j+\cdots+X_{m}^j.$$
Since $\tilde{R}_i\in\left(S_{\xi}[X_1,\ldots,\hat{X_i},\ldots,X_{m}]\right)^{\mathfrak{S}_{m-1}}$, it must therefore be a polynomial in the $P_{i,j}$'s with coefficients in $S_\xi$.  Note that for each $j$ there are polynomials $\tilde{P}_j(T)\in S_\xi[X_1,\ldots,X_m][T]$ such that for each $1\leq i\leq m$, $\tilde{P}_j(X_i)=P_{i,j}\in S_\xi[X_1,\ldots,\hat{X_i},\ldots,X_{m}]$, e.g.
$$\tilde{P}_j(T)\coloneqq X_1^j+\cdots+X_{m}^j-T^j.$$  
Hence there must be a polynomial $\tilde{Q}(T)\in S_\xi[X_1,\ldots,X_m][T]$ such that $\tilde{Q}(X_i)=\tilde{R}_i$ for each $1\leq i\leq m$.  Now we can take $Q(T)\in S_\xi[T]$ to be the image of $\tilde{Q}(T)$ under the evaluation map $X_i\mapsto\beta_i$.
\end{proof}

\subsection{Flip-Flop Maps}
\label{subsec:two}
Now we fix two $\phi$-regular values with the unique vertex $p\in V_\Gamma$ in between them, say $c<\phi(p)<c'$, and let $\Dc\coloneqq \bar{E}^p_-\subseteq V_c$ and $\Dcp\coloneqq E^p_+\subseteq V_{c'}$ as above.  Suppose that $\ind_\xi(p)=r$ so that $\left|\Dc\right|=r$ and $\left|\Dcp\right|=d-r\eqqcolon s$.  Recall that by Corollary \ref{cor:Lem4.1}, the ring $H(\Delta_p^{+})$, resp. $H(\Delta_p^-)$, is a free $S$ module generated by $\left\{\tau_{p+}^{i}\left|0\leq i\leq r-1\right.\right\}$, resp. $\left\{\tau_{p-}^a\left|0\leq a\leq s-1\right.\right\}$.  Define the \emph{down transition maps}, resp. \emph{up transition maps}, by
$$\xymap{H(\Dc)\ar[r]^-{\delta_0^p} & H(\Dcp)\\ \sum_{i=0}^{r-1}b_i\taupp^i\ar@{|->}[r] & \sum_{i=0}^{r-1}b_i\taupm^i\\} \ \text{resp.} \ \xymap{H(\Dcp)\ar[r]^-{\mu^p_0} & H(\Dc)\\ \sum_{a=0}^{s-1}c_a\taupm^a\ar@{|->}[r] & \sum_{a=0}^{s-1}c_a\taupp^a.\\}$$
Using Proposition \ref{prop:Thm6.1} and the identification of the sets $V_c\setminus\Dc$ and $V_{c'}\setminus \Dcp$, we can extend these transition maps to maps on $H(\Gamma_{c})$ and $H(\Gamma_{c'})$, i.e.
\begin{align}
\label{eq:downmap}
\xymap{H(\Gamma_{c'})\ar[r]^-{\delta^p} & \Maps(V_{c},S_\xi)\\ f\ar@{|->}[r] & {x\mapsto \begin{cases} f(x) & \text{if $x\in V_{c}\setminus\Dc$}\\ \delta_0^p\left(r_{c'}^-(f)\right)(x) & \text{if $x\in\Dc$}\\ \end{cases}.}}\\
\label{eq:upmap}
\xymap{H(\Gamma_{c})\ar[r]^-{\mu^p} & \Maps(V_{c'},S_\xi)\\ g\ar@{|->}[r] & {x\mapsto \begin{cases} g(x) & \text{if $x\in V_{c'}\setminus\Dcp$}\\ \mu_0^p\left(r_c^+(g)\right)(x) & \text{if $x\in\Dcp$}\\ \end{cases}.}}
\end{align}
It turns out that, as one might hope, the image of these maps each lie inside the corresponding cross sectional equivaraint cohomology ring.
\begin{lemma}
\label{lem:Flip}
For each $h\in H(\Gamma,\alpha)$, and for any fixed $f\in H(\Gamma_c)$ and $f'\in H(\Gamma_{c'})$ there are polynomials $F_p=F_p(x,\y), G_p=G_p(x,\y)\in S_\xi[x]\cong S$ such that 
\begin{align}
\label{eq:IntGlueForward}
\int_{\Gamma_{c'}}\mu^p(f)\Kircp(h)-\int_{\Gamma_c}f\Kirc(h)= & \frac{1}{c_p}\Res_\xi\left(\frac{F_p\cdot h_p}{\prod_{e\in E^p}\alpha(e)}\right)\\
\nonumber & \\
\label{eq:IntGlueBackward}
\int_{\Gamma_{c'}}f'\Kircp(h)-\int_{\Gamma_c}\delta^p(f')\Kirc(h)= & \frac{1}{c_p}\Res_\xi\left(\frac{G_p\cdot h_p}{\prod_{e\in E^p}\alpha(e)}\right).
\end{align}
In particular, $\mu^p(f)\in H(\Gamma_{c'})$ and $\delta^p(f')\in H(\Gamma_c)$.
\end{lemma}
\begin{proof}
We prove that \eqref{eq:IntGlueForward} holds.  The proof of \eqref{eq:IntGlueBackward} is analogous.  Fix $f\in H(\Gamma_c)$.  By Proposition \ref{prop:Thm6.1}, its restriction $r_+^c(f)$ lies in $H(\Dc)$, and thus by Lemma \ref{lem:Lem4.1} we can find unique $c_0,\ldots,c_{r-1}\in S_\xi$ such that $r_+^c(f)=\sum_{j=0}^{r-1}c_j\taupp^j$.  Hence the restriction of $\mu^p(f)$ to $\Dcp$ is given by 
$$r_-^c\left(\mu^p(f)\right)\coloneqq \sum_{j=0}^{r-1}c_j\taupm^j.$$  
Since $f\cdot\Kirc(h)$ and $\mu^p(f)\cdot\Kircp(h)$ agree on $V_c\setminus\Dc=V_{c'}\setminus\Dcp$, and since \eqref{eq:id1} holds, the difference on the LHS of \eqref{eq:IntGlueForward} becomes
\begin{equation}
\label{eq:g1}
\sum_{e\in\Dcp}\frac{\left(\sum_{j=0}^{r-1}c_j\beta_{e}^j\right)\cdot\rho_{e}\left(h(i(e))\right)}{c_pm_{e}\prod_{\substack{e'\in E^p\\e'\neq e\\}}m_{e'}(\beta_{e}-\beta_{e'})} +\sum_{e\in\Dc}\frac{\left(\sum_{j=0}^{r-1}c_j\beta_e^j\right)\cdot\rho_e\left(h(i(e))\right)}{c_pm_{e}\prod_{\substack{e'\in E^p\\e'\neq\bar{e}\\}}m_{e'}(\beta_e-\beta_{e'})}.
\end{equation}
Define the polynomial
$$F_p(x,\y)\coloneqq \sum_{j=0}^{r-1}c_jx^j.$$
Then $\rho_{e}(F_p(x,\y))=F_p(\beta_e,\y)=\sum_{j=0}^{r-1}c_j\beta_e^j$ for all $e\in E^p$, and \eqref{eq:g1} becomes
\begin{equation}
\label{eq:g3}
\frac{1}{c_p\prod_{e\in E^p}m_e}\sum_{e\in E^p}\frac{\rho_{e}(F_p(x,\y))\cdot\rho_{e}(h_p(x,\y))}{\prod_{\substack{e'\in E^p\\e'\neq e\\}}(\beta_e-\beta_{e'})}
\end{equation}
which, according to \eqref{eq:Res1}, is exactly the RHS of \eqref{eq:IntGlueForward}.
\end{proof} 

We call the resulting $S$-module maps
$$\xymap{H(\Gamma_c)\ar@/^/[r]^-{\mu_p} & H(\Gamma_{c'})\ar@/^/[l]^-{\delta_p}}$$ 
the \emph{flip-flop maps}.  

\subsection{Surjectivity of the Kirwan Map}
We can use our flip-flop maps and Lemma \ref{lem:Flip} to see that the Kirwan maps are surjective.
\label{subsec:three}
\begin{proposition}
\label{prop:SurjKirwan}
For each $\phi$-regular value $c$, the Kirwan map $\Kirc\colon H(\Gamma,\alpha)\rightarrow H(\Gamma_c)$ is surjective.
\end{proposition}
\begin{proof}
Fix regular values $c_0<\phi(p_1)<\cdots<c_{N-1}<\phi(p_N)<c_N$.  Fix $1\leq i\leq N-1$, and fix $f_i\in H(\Gamma_{c_i})$.  We can complete $f_i$ to a sequence of $\left\{f_k\right\}_{k=0}^{N}$ where $f_0\coloneqq 0\eqqcolon f_N$, $f_k\coloneqq\delta^{p_k}\circ\cdots\circ\delta^{p_i}(f_i)$ for $1\leq k<i$, and $f_k=\mu^{p_{k-1}}\circ\cdots\circ\mu^{p_i}(f_i)$ for $i< k\leq N-1$.  Then, by Lemma \ref{lem:Flip}, we have that $f_k\in H(\Gamma_{c_k})$ for all $1\leq k\leq N$ and we have 
\begin{equation}
\label{eq:SurjKirwan}
\int_{\Gamma_{c_k}}f_k\Kir_{c_k}(h)-\int_{\Gamma_{c_{k-1}}}f_{k-1}\Kir_{c_{k-1}}(h)=\frac{1}{c_{p_{k}}}\Res_\xi\left(\frac{F_{p_k}\cdot h_{p_k}}{\prod_{e\in E^{p_k}}\alpha(e)}\right),
\end{equation}
for some $F_{p_k}\in S_\xi[x]$.  We claim that the map $p\mapsto F_p$ is an equivariant class in $H(\Gamma,\alpha)$.  To see this, we add all the equations in \eqref{eq:SurjKirwan} to obtain
\begin{equation}
\label{eq:SurjKirwan2}
0=\sum_{k=1}^N\frac{1}{c_{p_k}}\Res_\xi\left(\frac{F_{p_k}\cdot h_{p_k}}{\prod_{e\in E^{p_k}}\alpha(e)}\right)=\Res_\xi\left(\sum_{p\in V_\Gamma}\frac{F_p\cdot h_p}{c_p\prod_{e\in E^p}\alpha(e)}\right),
\end{equation}
which holds for \emph{all} equivariant classes $h\in H(\Gamma,\alpha)$.  In particular, for any fixed $h\in H(\Gamma,\alpha)$ and for any fixed positive integer $k$, the element $h\cdot x^k$ is also in $H(\Gamma,\alpha)$.  In particular, we must have 
$$\Res_\xi\left(\sum_{p\in V_\Gamma}\frac{F_p\cdot h_p\cdot x^k}{c_p\prod_{e\in E^p}\alpha(e)}\right)=0$$
which implies, by Proposition \ref{prop:R2}, that $\sum_{p\in V_\Gamma}\frac{F_p\cdot h_p}{c_p\prod_{e\in E^p}\alpha(e)}\in S_\xi[x]\cong S$.  Since this holds for every $h\in H(\Gamma,\alpha)$, Proposition \ref{prop:StrDuality} implies that the map $F\colon p\mapsto F_p$ must also lie in $H(\Gamma,\alpha)$.  Summing the first $i$ equations in \eqref{eq:SurjKirwan} then yields
\begin{align}
\label{eq:SurjKirwan3}
\int_{\Gamma_{c_i}}f_i\Kir_{c_i}(h)= & \sum_{\phi(q)<c_i}\frac{1}{c_q}\Res_\xi\left(\frac{F_q\cdot h_q}{\prod_{e\in E^q}\alpha(e)}\right)\\
\nonumber = & \int_{\Gamma_{c_i}}\Kir_{c_i}(F)\cdot\Kir_{c_i}(h).
\end{align}
Since \eqref{eq:SurjKirwan3} holds for all $h\in H(\Gamma,\alpha)$ it follows that $\Kir_{c_i}(F)=f_i$, and hence $\Kir_{c_i}$ is surjective as claimed.
\end{proof}

\begin{corollary}
\label{cor:HGRing}
The submodule $H(\Gamma_c)\subseteq\Maps(V_c,S_\xi)$ is a subring, i.e. it is multiplicitively closed in $\Maps(V_c,S_\xi)$.
\end{corollary}
\begin{proof}
The Kirwan map $\Kirc$ is a ring homomorphism onto its image.  By Proposition \ref{prop:SurjKirwan} its image is $H(\Gamma_c)$.
\end{proof}

\subsection{Proof of Theorem \ref{thm:GZMain}}
The following lemma will make short work of the proof of Theorem \ref{thm:GZMain}.  First we fix some notation.

Fix a $\phi$-regular value $c$, and let $p\in V_\Gamma$ be the largest vertex such that $\phi(p)<c$.  Let $r\coloneqq\ind_\xi(p)$.  Label the edges at $p$, $E^p=\left\{e_1,\ldots,e_r,e_{r+1},\ldots,e_d\right\}$ such that $\Delta_c^-\coloneqq\left\{e_{r+1},\ldots,e_d\right\}$.  We use the shorthand $\beta_\ell=\beta_{e_\ell}\in S_\xi$.

\begin{lemma}
\label{lem:SurjKMorse}
If the Kirwan map $\Kirc\colon H(\Gamma,\alpha)\rightarrow H(\Gamma_c)$ is surjective, then the vertex $p$ has a weak generating class.
\end{lemma}
\begin{proof}
Suppose $\Kirc\colon H(\Gamma,\alpha)\rightarrow H(\Gamma_c)$ is surjective.  Consider the map $F\colon V_c\rightarrow S_\xi$ defined by 
$$F(e_a)=\begin{cases} \prod_{i=1}^r\beta_a-\beta_i & \text{if $e_a\in \Delta_c^-$}\\
0 & \text{otherwise}.\\
\end{cases}$$
Note that for any $h\in H(\Gamma,\alpha)$ we have
\begin{align*}
\int_{\Gamma_c}F\cdot\Kirc(h)= & \sum_{a=r+1}^d\frac{\prod_{i=1}^r(\beta_a-\beta_i)\cdot \rho_{e_a}\left(h_p(x,\y)\right)}{c_{p}m_a\prod_{\substack{j=1\\ j\neq a\\}}^dm_j(\beta_a-\beta_j)}\\
= & \frac{1}{c_pM_p}\sum_{a=r+1}^d\frac{h_p(\beta_a,\y)}{\prod_{\substack{j=r+1\\ j\neq a\\}}^d(\beta_a-\beta_j)}\\
= & \frac{1}{c_p}\Res_\xi\left(\frac{h_p(x,\y)}{\prod_{j=r+1}^d\alpha(e_a)}\right).
\end{align*}
In particular, we see that $F\in H(\Gamma_c)$.  By our surjectivity assumption, there is some homogeneous equvariant class $T\in H(\Gamma,\alpha)$ of degree $r$ such that $F=\Kirc(T)$.  Now define $\tau_p\colon V_\Gamma\rightarrow S$ by the rule
\begin{equation}
\label{eq:pGenClass}
\tau_p(x)=\begin{cases}
\prod_{i=1}^r\alpha(e_i) & \text{if $x=p$}\\
M_pT(x) & \text{if $\phi(x)>\phi(p)$}\\
0 & \text{if $\phi(x)<\phi(p)$},\\
\end{cases}
\end{equation}
where as above $M_p\coloneqq\prod_{i=1}^rm_{e_i}$.
Clearly if $\tau_p$ were an equivariant class in $H(\Gamma,\alpha)$ it would have to be a weak generating class for $p$.  Therefore it suffices to show that for each $xy\in E_\Gamma$ there is some $c_{xy}\in S$ such that 
\begin{equation}
\label{eq:CC}
\tau_p(y)-\tau_p(x)=c_{xy}\cdot\alpha(xy).
\end{equation} 
Fix any oriented edge $xy\in E_\Gamma$.  We may assume without loss of generality that $\phi(x)<\phi(y)$.  Certainly \eqref{eq:CC} is satisfied for some $c_{xy}\in S$ if $\phi(x)<\phi(y)<\phi(p)$ or if $\phi(p)<\phi(x)<\phi(y)$.  If $\phi(x)<\phi(p)<\phi(y)$ then $xy\in V_c\setminus\Delta_c^-$, and we have 
$$\tau_p(y)-\tau_p(x)=\tau_p(y)=M_pT(y).$$
Since $\Kirc(T)(xy)=\rho_{xy}(T(y))=0$, we see that $T(y)$ must indeed be a multiple of $\alpha(xy)$ since it is in the kernel of the map $\rho_{xy}\colon S\rightarrow S_\xi$. 
If $\phi(y)=\phi(p)$ then $xy\in\left\{e_1,\ldots,e_r\right\}$, then $\tau_p(y)-\tau_p(x)=\tau_p(p)=\prod_{i=1}^r\alpha(e_i)$ is a multiple of $\alpha(xy)$.
Finally if $\phi(x)=\phi(p)$ then $xy\in\Delta_c^-$, and we have 
$$\tau(y)-\tau(x)=M_pT(y)-\prod_{i=1}^r\alpha(e_i).$$
In this case notice that 
$$\rho_{xy}(M_pT(y))=M_p\prod_{i=1}^r\beta_{xy}-\beta_i=\prod_{i=1}^rm_{e_i}(\beta_{xy}-\beta_i)=\rho_{xy}\left(\prod_{i=1}^r\alpha(e_i)\right)$$
which again implies that the difference $\tau_p(y)-\tau_p(x)$ is a multiple of $\alpha(xy)$.  
Thus $\tau_p$ is a class in $H(\Gamma,\alpha)$, as claimed.
\end{proof}

\begin{proof}[Proof of Theorem \ref{thm:GZMain}]
Assume that $(\Gamma,\alpha,\theta,\lambda)$ has the Morse package.  Fix a polarizing covector $\xi\in\left(\R^n\right)^*$, a compatible Morse function $\phi\colon V_\Gamma\rightarrow\R$, and a generating family $\left\{\tau_p\right\}_{p\in V_\Gamma}$.  As we have already noted for the proof of Proposition \ref{prop:Thm6.1}, $\xi$, resp. $\phi$, restricts to a polarizing covector, resp. a Morse function, on every $2$-slice.  Since the restriction of an equivariant class to a subskeleton is obviously an equivariant class of the subskeleton, we deduce that the restriction of the generating class to a $2$-slice is a generating class for that $2$-slice.  Thus every $2$-slice inherits the Morse package from $(\Gamma,\alpha,\theta,\lambda)$.

Now assume that every $2$-slice of $(\Gamma,\alpha,\theta,\lambda)$ has the Morse package.  Then by Proposition \ref{prop:SurjKirwan}, the Kirwan maps $\Kirc\colon H(\Gamma,\alpha)\rightarrow H(\Gamma_c)$ are surjective.  Then Lemma \ref{lem:SurjKMorse} implies that every vertex $p\in V_\Gamma$ has a weak generating class.  Thus by Proposition \ref{prop:MorsePackage}, we deduce that $(\Gamma,\alpha,\theta,\lambda)$ has the Morse package.
\end{proof}

\section{Concluding Remarks}
\label{sec:Comments}
Theorem \ref{thm:GZMain} says that in order to understand 1-skeleta with the Morse package, it is enough to look at 1-skeleta in $\R^2$ or \emph{planar 1-skeleta}.  It is an open problem to classify those planar 1-skeleta which have the Morse package.  Here are a few remarks in this direction.

A planar 1-skeleton is called \emph{noncyclic} if it satisfies the acyclicity axiom and is pointed.  It follows from Proposition \ref{prop:StMor} and the discussion preceding it that a $d$-valent planar 1-skeleton with the Morse package must be noncyclic and straight.  Note that any noncyclic 1-skeleton has a generating class in degree zero, namely the constant map $p\mapsto 1$.  Moreover, since vertices and edges support Thom classes, any non-cyclic 1-skeleton must also admit generating classes in degrees $d$ and $d-1$.  It turns out that if our non-cyclic 1-skeleton is straight and planar then we can do slightly better.
\begin{proposition}
\label{prop:d2valent}
Every noncyclic planar straight 1-skeleton admits generating classes in degree $d-2$.  
\end{proposition}
\begin{proof}
Let $(\Gamma,\alpha,\theta,\lambda)$ be a noncyclic planar straight 1-skeleton, let $\xi\in\left(\R^2\right)^*$ be any polarizing covector, let $p\in V_\Gamma$ be any vertex of index $d-2$, and let $\mathcal{F}_p$ denote the flow up at $p$.  Then there exists a $2$-valent subgraph $\Gamma_p=(V_p,E_p)$ containing $p$ such that $V_p\subseteq\mathcal{F}_p$.  Indeed to find such a graph $\Gamma_p$, one can simply take two oriented paths starting from $p$ (which exist since $\ind_\xi(p)=d-2$) and follow them until they meet (which they must since the orientation on $\Gamma$ must have a unique sink by the noncyclic assumption).

Label the vertices $V_p$ in cyclic order, say $v_1,\ldots,v_n$, so that $v_1=p$ and $v_iv_{i+1}\in E_p$.  For each $x\in V_p$, set $N_p^x$ to be the oriented edges at $x$ normal to $\Gamma_p$.  Then for each $1\leq i\leq n$ we have 
\begin{equation}
\label{eq:d21}
\prod_{\substack{e\in E^{v_i}\\e\neq v_{i}v_{i+1}\\}}\alpha(e)\equiv \left|K_{v_iv_{i+1}}\right|\cdot\prod_{\substack{e\in E^{v_{i+1}}\\e\neq v_{i+1}v_i\\}}\alpha(e) \ \ \ mod \ \alpha(v_iv_{i+1}).
\end{equation}
Note that $\alpha(v_iv_{i+1})$ and $\alpha(v_iv_{i-1})$ are a basis for $\R^2$, hence for each $1\leq i\leq n$ we can find real numbers $\lambda_i$ such that
\begin{equation}
\label{eq:d22}
\alpha(v_iv_{i-1})\equiv\lambda_i\alpha(v_{i+1}v_{i+2}) \ \ \ mod \ \alpha(v_iv_{i+1}).
\end{equation}
Dividing these two congruences we get that 
\begin{equation}
\label{eq:d23}
\prod_{e\in N^{v_i}_p}\alpha(e)\equiv\frac{\left|K_{v_iv_{i+1}}\right|}{\lambda_i}\prod_{e\in N^{v_{i+1}}_p}\alpha(e).
\end{equation}
Note that from \eqref{eq:d22}, the $\lambda_i$'s can be computed as quotients of exterior products, i.e.
\begin{equation}
\label{eq:d24}
\lambda_i=\frac{\alpha(v_iv_{i-1})\wedge\alpha(v_iv_{i+1})}{\alpha(v_{i+1}v_i)\wedge\alpha(v_{i+1}v_{i+2})}
\end{equation}
where $x\wedge y$ denotes the product in the exterior algebra $\bigwedge(\R^2)$.  From \eqref{eq:d24} we see that the product $\lambda_1\cdots\lambda_n$ must equal one.  Hence by straightness, we must also have  
\begin{equation}
\label{eq:d25}
\prod_{i=1}^n\frac{\left|K_{v_iv_{i+1}}\right|}{\lambda_i}=1.
\end{equation}
Now define $M_1\coloneqq 1$ and for $2\leq i\leq n$, define 
$$M_i\coloneqq \prod_{j=1}^{i-1}\frac{\left|K_{v_jv_{j+1}}\right|}{\lambda_j},$$
and define the map $\tau_p\colon V_\Gamma\rightarrow S$ by the formula
$$\tau_p(x)=\begin{cases}
M_i\prod_{e\in N_p^{v_i}}\alpha(e) & \text{if $x=v_i$}\\
0 & \text{otherwise}\\
\end{cases}$$
To finish the proof, we need only show that $\tau_p\in H(\Gamma,\alpha)$, i.e. for every oriented edge $xy\in E_\Gamma$
\begin{equation}
\label{eq:d26}
\tau_p(y)-\tau_p(x)\equiv 0 \ \ \ mod \ \alpha(xy).
\end{equation}
There are three cases to consider:  $xy\in E_p$, $xy\in N_p$, or neither.  If $xy$ is in neither $E_p$ nor $N_p$, then the difference on the LHS of \eqref{eq:d26} is zero, hence the equivalence is satisfied.  If $xy\in N_p$ then both $\tau_p(x)$ and $\tau_p(y)$ are multiples of $\alpha(xy)$ and again the equivalence \eqref{eq:d26} is satisfied.  Finally if $xy\in E_p$ then we may assume that $x=v_i$ and $y=v_{i+1}$ for some $1\leq i\leq N$, where of course $v_{N+1}\coloneqq v_1$.  The key observation to make here is that for every $1\leq i\leq N$ we have 
$$M_{i+1}=\frac{\left|K_{v_iv_{i+1}}\right|}{\lambda_i}\cdot M_i.$$  
Thus the RHS of \eqref{eq:d26} can be written
\begin{equation}
\label{eq:d27}
M_i\left(\frac{\left|K_{v_iv_{i+1}}\right|}{\lambda_i}\prod_{e\in N_p^{v_{i+1}}}\alpha(e)-\prod_{e\in N_p^{v_i}}\alpha(e)\right),
\end{equation} 
and of course the equivalence in \eqref{eq:d26} now follows from the equivalence in \eqref{eq:d23}.  Therefore $\tau_p$ is a generating class for $p$.
\end{proof}
In particular Proposition \ref{prop:d2valent} allows a nice characterization of the Morse package for planar 1-skeleta with small valency.  
\begin{corollary}
\label{cor:3valent}
Every noncyclic planar straight $3$-valent 1-skeleton has the Morse package.
\end{corollary}
  
One might naively guess that noncyclic and straight are sufficient conditions for the Morse package in higher valencies, but this is not the case as the following example shows.

Let $P\subseteq\R^2$ be a regular $7$-gon with vertices labeled in cyclic order $p_1,\ldots,p_7$.  Let $\Gamma=(V_\Gamma,E_\Gamma)$ be the graph with $V_\Gamma=\left\{p_1,\ldots,p_7\right\}$ and edges $E_\Gamma=\left\{p_ip_{i\pm1}, \ p_ip_{i\pm 3}\right\}$ where the indices are understood modulo $7$.  For each $1\leq i\leq 7$ let $\p_i$ denote the position vector of vertex $p_i$.  An axial function $\alpha\colon E_\Gamma \rightarrow\R^2$ is then defined by $\alpha(p_ip_j)\coloneqq \p_j-\p_i$.  Reflections across the edges of $\Gamma$ define a connection $\theta$ for the pair $(\Gamma,\alpha)$ and with this connection, one can check the the compatibility constants are all equal to one.  Hence the resulting $4$-valent planar 1-skeleton $(\Gamma,\alpha,\theta,\lambda)$ is straight (it is GKM) and also noncyclic, since its vertices are in convex position.  See Figure \ref{fig:cr}.  Note however that $(\Gamma,\alpha,\theta,\lambda)$ cannot have the Morse package.  Indeed if it did have the Morse package then, taking $\xi=(0,1)\in\left(\R^2\right)^*$, a generating class for $p_5$ would need to be a degree one class supported on vertices $p_1,\ldots,p_5$.  On the other hand, any degree one class in $H(\Gamma,\alpha)$ that is zero on $p_6$ and $p_7$ must be identically zero.  
\begin{figure}
  \includegraphics{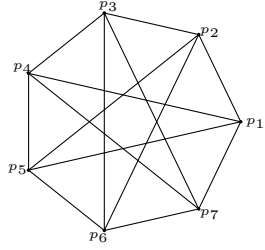}
\caption{No Morse Package}
\label{fig:cr}       
\end{figure}

The Morse package on a 1-skeleton implies that its equivariant cohomology is a free module over a polynomial ring.  It seems natural to ask for necessary and sufficient conditions for the equivariant cohomology of a 1-skeleton to be a free module over the polynomial ring $S$ in general.  Recently Luo \cite{Luo} has proved that the equivariant cohomology of any planar 1-skeleton is always free.  Due to Luo's result and perhaps a lack of counter examples, one might venture to guess that the equivariant cohomology of any 1-skeleton in $\R^n$ is always a free module over the polynomial ring $\Sym(\R^n)$.

\bibliographystyle{plain}
\bibliography{morse}

\end{document}